\documentclass[12pt]{amsart}
\usepackage{amsmath,amsfonts,amssymb,amsthm,mathrsfs,a4wide,hyperref,rotating}
\usepackage[all]{xy}
\xyoption{tips}

\SelectTips{cm}{12}

\hypersetup{colorlinks=true}

\theoremstyle{plain}
\newtheorem{theorem}{Theorem}[section]
\newtheorem{proposition}[theorem]{Proposition}
\newtheorem{lemma}[theorem]{Lemma}
\newtheorem{corollary}[theorem]{Corollary}
\newtheorem{conjecture}[theorem]{Conjecture}
\newtheorem*{theorem*}{Theorem}

\theoremstyle{definition}
\newtheorem{definition}[theorem]{Definition}
\newtheorem{example}[theorem]{Example}
\newtheorem{remark}[theorem]{Remark}
\newtheorem{question}[theorem]{Question}
\newtheorem{notation}[theorem]{Notation}

\newcommand{\KM}{\mathfrak{K}(M)}

\newcommand{\res}[1]{\!\downarrow_{#1}}
\newcommand{\leftind}[2]{{\vphantom{#2}}_{#1}{#2}}
\newcommand{\leftindexp}[3]{{\vphantom{#2}}_{#1}^{#2}{#3}}

\DeclareMathOperator{\Ker}{\mathsf{Ker}}
\DeclareMathOperator{\Img}{\mathsf{Im}}
\DeclareMathOperator{\Coker}{\mathsf{Coker}}
\DeclareMathOperator{\Rad}{\mathsf{Rad}}

\DeclareMathOperator{\fgmod}{\mathsf{mod}}
\DeclareMathOperator{\rank}{rank}
\DeclareMathOperator{\cJt}{\mathsf{cJt}}
\DeclareMathOperator{\vect}{\mathsf{vec}}
\DeclareMathOperator{\Proj}{\mathsf{Proj}}
\DeclareMathOperator{\coh}{\mathsf{coh}}
\DeclareMathOperator{\Fun}{\mathsf{Fun}}
\DeclareMathOperator{\Tor}{\mathsf{Tor}}
\DeclareMathOperator{\JType}{\mathsf{JType}}

\begin{document}

\title[Constant Jordan type, pullbacks of bundles and generic kernels]
{Modules of constant Jordan type, pullbacks of bundles and generic
kernel filtrations}
\author{Shawn Baland and Kenneth Chan}
\address{Department of Mathematics, University of Washington, Seattle,
Washington 98185}

\begin{abstract}
Let $kE$ denote the group algebra of an elementary abelian $p$-group
of rank $r$ over an algebraically closed field of characteristic $p$.
We investigate the functors $\mathcal{F}_i$ from $kE$-modules of
constant Jordan type to vector bundles on $\mathbb{P}^{r-1}(k)$,
constructed by Benson and Pevtsova. For a $kE$-module $M$ of constant
Jordan type, we show that restricting the sheaf $\mathcal{F}_i(M)$ to
a dimension $s-1$ linear subvariety of $\mathbb{P}^{r-1}(k)$ is equivalent
to restricting $M$ along a corresponding rank $s$ shifted subgroup of $kE$
and then applying $\mathcal{F}_i$.

In the case $r=2$, we examine the generic kernel filtration of $M$
in order to show that $\mathcal{F}_i(M)$ may be computed on certain
subquotients of $M$ whose Loewy lengths are bounded in terms of $i$.
More precise information is obtained by applying similar techniques to
the $n$th power generic kernel filtration of $M$. The latter approach
also allows us to generalise our results to higher ranks $r$.
\end{abstract}

\maketitle

\tableofcontents

\section{Introduction}

The goal of this paper is to further investigate a curious functorial
relationship between the category of finitely generated $kE$-modules
and the category of coherent sheaves on the projective space
$\mathbb{P}^{r-1}(k)$, where $E$ is an elementary abelian $p$-group of
rank $r$ and $k$ is an algebraically closed field of characteristic $p$.
Specifically, we wish to better understand the functors
\[
\xymatrix{
\mathcal{F}_i\colon\fgmod(kE) \ar@{>}[r] & \coh(\mathbb{P}^{r-1}(k)),
\qquad 1\leq i\leq p
}
\]
introduced by Benson and Pevtsova \cite{Benson/Pevtsova:vb}.
Interest in the functors $\mathcal{F}_i$ originated in the study
of $kE$-modules of constant Jordan type, which were defined by
Carlson, Friedlander and Pevtsova
\cite{Carlson/Friedlander/Pevtsova:2008a}. Denoting the subcategory of
modules of constant Jordan type by $\cJt(kE)$, Benson and Pevtsova
showed that the functors $\mathcal{F}_i$ descend to functors
\[
\xymatrix{
\mathcal{F}_i\colon\cJt(kE) \ar@{>}[r] & \vect(\mathbb{P}^{r-1}(k)),
}
\]
where $\vect(\mathbb{P}^{r-1}(k))$ is the category of vector bundles
on $\mathbb{P}^{r-1}(k)$. We remark that neither of the latter two
categories is well understood. Whereas the study of modules of constant
Jordan type is a relatively new enterprise, the attempt to understand
what sorts of vector bundles can live on $\mathbb{P}^{r-1}(k)$ has been
ongoing since the advent of modern algebraic geometry, and with limited
success. Accordingly, a thorough understanding of the functors
$\mathcal{F}_i$ should be of interest to representation theorists and
algebraic geometers alike.

In this direction, our aim is to further establish some sort of
dictionary between modules of constant Jordan type and vector bundles
on $\mathbb{P}^{r-1}(k)$ via the functors $\mathcal{F}_i$. For
example, one of the common techniques of the algebraic geometer is that
of restricting a vector bundle on $\mathbb{P}^{r-1}(k)$ to a line
$L$ in $\mathbb{P}^{r-1}(k)$ in order to compute its so called
`splitting type'. Any such closed immersion
$L\subseteq\mathbb{P}^{r-1}(k)$ is obtained by applying the $\Proj$
functor to a surjective homogeneous ring homomorphism
\[
\xymatrix{
k[Y_1,\ldots,Y_r] \ar@{>}[r] & k[Z_1,Z_2].
}
\]
In Section \ref{pullbacksection} we generalise this situation a bit and
show that any surjective ring homomorphism of the form
\[
\xymatrix{
k[Y_1,\ldots,Y_r] \ar@{>}[r] & k[Z_1,\ldots,Z_s],\qquad s\leq r
}
\]
arises naturally from what we call a
\emph{homogeneously embedded $s$-shifted subgroup} $kE'$ of $kE$.
Our main result related to this is that, under the functors
$\mathcal{F}_i$, pulling back along such a closed immersion
$\mathbb{P}^{s-1}(k)\hookrightarrow\mathbb{P}^{r-1}(k)$ corresponds
to restricting scalars along the inclusion $kE'\hookrightarrow kE$.
Specifically, we obtain the following.

\begin{theorem*}
Let $kE'$ be a homogeneously embedded $s$-shifted subgroup of $kE$
and let
\[
\xymatrix{
f\colon\mathbb{P}^{s-1}(k) \ar@{>}[r] & \mathbb{P}^{r-1}(k)
}
\]
be the corresponding closed immersion. If $M$ is a $kE$-module of
constant Jordan type, then for all $1\leq i\leq p$ we have
$f^*\mathcal{F}_i(M)\cong\mathcal{F}_i(M\res{kE'})$.
\end{theorem*}

Here, the notation has been abused slightly so that $\mathcal{F}_i$
denotes both functors on $\fgmod(kE)$ and $\fgmod(kE')$, respectively.

In connection with the long term goal of studying $kE$-modules of
constant Jordan type by looking at splitting types of vector bundles
on $\mathbb{P}^{r-1}(k)$, the remainder of this paper will be dedicated to
better understanding the behaviour of modules of constant Jordan type
in the case $r=2$, so that the bundles $\mathcal{F}_i(M)$ live over
$\mathbb{P}^1(k)$. The representation theory of $kE$ in the case $r=2$
was closely investigated by Carlson, Friedlander and Suslin
\cite{Carlson/Friedlander/Suslin:2011a}. In that paper, the authors
constructed an interesting functorial invariant of a $kE$-module called
the \emph{generic kernel}. If $M$ is any finite dimensional $kE$-module in
rank two, then its generic kernel $\KM$ can be characterised as the
largest submodule of $M$ having the so called `equal images property'.
(See Definition \ref{eipdefinition}.) The generic kernel gives rise to
a filtration of $M$ whose terms are $J^i\KM$, where $J=J(kE)$
denotes the Jacobson radical of $kE$, and for $j>0$, $J^{-j}\KM$ denotes
the collection of elements $m\in M$ for which $J^jm\subseteq\KM$.
In the first author's \cite{Baland:2012a}, this was called the
\emph{generic kernel filtration}.

If $M$ has something called the `constant rank' property, which is a
relatively mild condition, then the generic kernel filtration of $M$ has
some interesting features. For example, its filtered quotients are
semisimple, and there is a well behaved duality theory.
(See Lemma \ref{generickernelduality}.)
Our results here will show further that for a $kE$-module $M$ of constant
Jordan type, the various layers of the generic kernel filtration allow
one to compute the vector bundles $\mathcal{F}_i(M)$ on what are
generally much smaller subquotients of $M$. In their fullest,
our results establish the following.

\begin{theorem*}
If $r=2$ and $M$ is a $kE$-module of constant Jordan type, then for
each $1\leq i\leq p$ we have
\[
\mathcal{F}_i(M)\cong\mathcal{F}_i(J^{-i}\KM/J^{i+1}\KM).
\]
\end{theorem*}

Our main objective in reducing to the above subquotients when computing
$\mathcal{F}_i(M)$ was the following: Although the vector
bundles on $\mathbb{P}^1(k)$ may be described succinctly
(every bundle is a direct sum of twists of the structure sheaf), the
category of $kE$-modules having constant Jordan type is known to be
wild unless $p=2$. It would be satisfying to find a subcategory
of $\fgmod(kE)$ through which the functor $\mathcal{F}_i$ factors,
one which lends itself to some sort of structure theorem.
Unfortunately, even when computing $\mathcal{F}_1(M)$, the
above theorem deals with the subquotient $J^{-1}\KM/J^2\KM$. Although
such subquotients have Loewy length only three, the class of modules
of the form $J^{-1}\KM/J^2\KM$ (where $M$ has constant Jordan type)
remains wild. There must certainly be a better behaved structural invariant
of $M$ that determines what $\mathcal{F}_i(M)$ is. In other words,
the goal of our program is to determine exactly how much of the structure
of $M$ the functors $\mathcal{F}_i$ actually detect.

The search for such a structural invariant has led us to investigate
a different filtration, namely the $n$th power generic kernel filtration.
As with the generic kernel, the $n$th power generic kernel ($n\geq 1$) of a
$kE$-module $M$ was also introduced in
\cite{Carlson/Friedlander/Suslin:2011a}. We will show that the $n$th power
generic kernels, along with their duals, give rise to a filtration of $M$
that performs the same task with respect to computing $\mathcal{F}_i(M)$
as the regular generic kernel filtration does, but with two added benefits:
First, a suitable choice of definition allows us to generalise our results
to the case $r>2$, and second, the Loewy length three subquo-
tients in the
$n$th power generic kernel filtration on which $\mathcal{F}_1(M)$ is
computed appear to have more tractable structures. In particular, it
appears that $\mathcal{F}_1(M)$ is determined by the Loewy length two
summands of these subquotients. We make a precise conjecture about this
point in our final section.

The authors would like to thank Julia Pevtsova and Alexandru Chirvasitu
for numerous discussions that contributed to the development
of this work.

\section{Background on elementary abelian
\texorpdfstring{$p$}{p}-group representations}

Let $p$ be a prime number and $k$ an algebraically closed field of
characteristic $p$. Throughout this paper, $E\cong(\mathbb{Z}/p)^r$
will denote an elementary abelian $p$-group of rank $r$. In what follows,
we choose a collection of pairwise commuting generators
$g_1,\ldots,g_r$ of $E$.
Letting $X_i$ denote the element $g_i-1$ in the group algebra $kE$,
we have $X_i^p=(g_i-1)^p=0$ since $k$ has characteristic $p$. We may
therefore identify $kE$ with the truncated
polynomial ring $k[X_1,\ldots,X_r]/(X_1^p,\ldots,X_r^p)$. In particular,
$kE$ is a local ring, and the Jacobson radical of $kE$ is generated
by the elements $X_i$.

If $\alpha=(\lambda_1,\ldots,\lambda_r)$ is any point in the affine
space $\mathbb{A}^r(k)$, we define the element
\[
X_\alpha=\lambda_1X_1+\cdots+\lambda_rX_r\in kE.
\]
Observe that we have $X_\alpha^p=0$ for all $\alpha$.
The assignment $\alpha\mapsto X_\alpha$ allows one to identify
$\mathbb{A}^r(k)$ with $\Rad(kE)/\Rad^2(kE)$. Note that for each
non-zero $\alpha$, the element $1+X_\alpha\in kE$ has multiplicative
order $p$, hence it generates a subgroup of $kE^\times$ isomorphic to
$\mathbb{Z}/p$. The subalgebra $k\langle 1+X_\alpha\rangle\subseteq kE$
is called a \emph{cyclic shifted subgroup} of $kE$. The nomenclature
is designed to indicate that, although $k\langle 1+X_\alpha\rangle$
is a group algebra, $\langle 1+X_\alpha\rangle$ is not a subgroup of
$E$ in general.

Throughout this paper we will deal exclusively with finitely generated
(i.e., finite dimen-
sional) $kE$-modules. If $M$ is a $kE$-module
and $\alpha$ is a point in $\mathbb{A}^r(k)$, then because $X_\alpha^p=0$,
the Jordan canonical form of the matrix representing the action of
$X_\alpha$ on $M$ consists of Jordan blocks whose eigenvalues
are all zero and whose lengths are at most $p$. The \emph{Jordan type}
of $X_\alpha$ on $M$ is defined to be the partition
\[
\JType(X_\alpha,M)=[p]^{a_p}[p-1]^{a_{p-1}}\ldots[1]^{a_1}
\]
of $\dim_k(M)$, where $X_\alpha$ acts on $M$ via $a_j$ Jordan blocks
of length $j$. Recall that representations of $k(\mathbb{Z}/p)$ are
classified in terms of Jordan canonical forms. In that context,
if $\alpha\neq 0$, then the Jordan type of $X_\alpha$ on $M$ is
precisely the isomorphism type of $M\res{k\langle 1+X_\alpha\rangle}$
when viewed as a $k(\mathbb{Z}/p)$-module via the identification
$\langle 1+X_\alpha\rangle\cong\mathbb{Z}/p$.

At first, one might be tempted to try to classify the indecomposable
objects in $\fgmod(kE)$. Unfortunately, $kE$ has wild representation type
unless $r=1$ or $r=p=2$, which essentially makes that task impossible.
We therefore confine ourselves to identifying special subcategories of
$\fgmod(kE)$ that we hope to better understand in terms of certain
invariants. This paper is primarily concerned with the category of
modules of constant Jordan type, which were introduced by
Carlson, Friedlander and Pevtsova \cite{Carlson/Friedlander/Pevtsova:2008a}.

\begin{definition}
A $kE$-module $M$ has \emph{constant Jordan type} if the partition
$\JType(X_\alpha,M)$ is independent of the choice of non-zero
$\alpha\in\mathbb{A}^r(k)$. If $M$ has constant Jordan type and
$\JType(X_\alpha,M)=[p]^{a_p}\ldots[1]^{a_1}$ for all non-zero $\alpha$,
then we say that $M$ has \emph{constant Jordan type}
$[p]^{a_p}\ldots[1]^{a_1}$. We denote the full subcategory of modules
of constant Jordan type by $\cJt(kE)$.
\end{definition}

For our purposes, the main feature of modules of constant Jordan
type is that they give rise to vector bundles
(i.e., locally free coherent sheaves) on $\mathbb{P}^{r-1}(k)$
in a natural way. Let $V$ be the subspace of $kE$ spanned by
$X_1,\ldots,X_r$. For $1\leq i\leq r$, let $Y_i\in V^\#$
be the basis element dual to $X_i$. The $Y_i$ then act as homogeneous
coordinate functions on $V$, and we identify
$\mathbb{P}^{r-1}(k)$ with $\Proj k[Y_1,\ldots,Y_r]$.
For an arbitrary $kE$-module $M$, let $\widetilde{M}$ denote the coherent
sheaf $M\otimes_k\mathcal{O}_{\mathbb{P}^{r-1}}$. In
\cite{Friedlander/Pevtsova:constr}, Friedlader and Pevtsova introduced
the operators
\[
\xymatrix{
\theta_M\colon\widetilde{M}(n) \ar@{>}[r] & \widetilde{M}(n+1)
}
\]
defined locally as follows: Any section of $\widetilde{M}(n)$ is of the
from $m\otimes f$, where $m\in M$ and $f$ is a homogeneous rational
function of degree $n$ in $Y_1,\ldots,Y_r$. The map $\theta_M$ is defined
by mapping $m\otimes f$ to the section $\sum_i X_im\otimes Y_if$ of
$\widetilde{M}(n+1)$. The virtue of this setup is that if
$\overline{\alpha}\in\mathbb{P}^{r-1}(k)$ is a closed point and
$\alpha\in\mathbb{A}^r(k)$ is a point lying above $\overline{\alpha}$,
then the fibre of $\theta_M$ at $\overline{\alpha}$ recovers
(up to a scalar factor) the $k$-linear map $X_\alpha\colon M\rightarrow M$.

We now describe the functors of interest in this paper. For a
$kE$-module $M$ and $1\leq i\leq p$, Benson and Pevtosva
\cite{Benson/Pevtsova:vb} defined the coherent sheaves
\[
\mathcal{F}_i(M)=
\frac{
\Ker\theta_M\cap\Img\theta_M^{i-1}
}
{
\Ker\theta_M\cap\Img\theta_M^i
}.
\]
Here, $\Ker\theta_M$ denotes the kernel of the morphism
$\widetilde{M}\rightarrow\widetilde{M}(1)$, whereas, for $j=i-1$ and $i$,
$\Img\theta_M^j$ denotes the image of the morphism
$\widetilde{M}(-j)\rightarrow\widetilde{M}$. With these conventions,
$\mathcal{F}_i(M)$ is a subquotient of $\widetilde{M}$. The following,
which appeared in \cite{Benson/Pevtsova:vb}, is the main fact concerning
the functors $\mathcal{F}_i$.

\begin{proposition}
A $kE$-module $M$ has constant Jordan type $[p]^{a_p}\ldots[1]^{a_1}$
if and only if, for each $1\leq i\leq p$, the coherent sheaf
$\mathcal{F}_i(M)$ is a vector bundle of rank $a_i$ on
$\mathbb{P}^{r-1}(k)$.
\end{proposition}

The main result of \cite{Benson/Pevtsova:vb} showed that
the functor $\mathcal{F}_1$ realises all vector bundles on
$\mathbb{P}^{r-1}(k)$ up to a Frobenius twist. We record it here
in order to motivate our overall interest in the functors
$\mathcal{F}_i$, and our particular interest in the bahaviour of the
functor $\mathcal{F}_1$, the latter being the subject of our examples.

\begin{theorem}
If $p=2$ and $\mathcal{F}$ is a vector bundle of rank $s$ on
$\mathbb{P}^{r-1}(k)$, then there exists a $kE$-module of constant Jordan
type of the form $[p]^n[1]^s$ (for some $n$) such that
$\mathcal{F}_1(M)\cong\mathcal{F}$.
If $p>2$ and $\mathcal{F}$ is a vector bundle on $\mathbb{P}^{r-1}(k)$,
then there exists a $kE$-module of the form $[p]^n[1]^s$ such that
$\mathcal{F}_1(M)\cong F^*\mathcal{F}$, where
$F\colon\mathbb{P}^{r-1}(k)\rightarrow\mathbb{P}^{r-1}(k)$ is the
Frobenius morphism.
\end{theorem}

\section{Pullbacks of bundles and homogeneously embedded subgroups}
\label{pullbacksection}

The natural generalisation of a cyclic shifted subgroup of $kE$ is
a \emph{rank $s$ shifted subgroup} of $kE$, where $s\leq r$ is a fixed
positive integer. Specifically, a rank $s$ shifted subgroup is a
subalgebra of $kE$ that is isomorphic to the group algebra $kE'$, where
$E'$ is an elementary abelian $p$-group of rank $s$. Any embedding
$\phi\colon kE'\hookrightarrow kE$ is obtained by mapping a choice of
generators $T_1,\ldots,T_s$ of $\Rad(kE')$ to elements
$\phi(T_1),\ldots,\phi(T_s)\in\Rad(kE)$ whose images in
$\Rad(kE)/\Rad^2(kE)$ are linearly independent. In the following,
we restrict attention to the embeddings $\phi$ for which the elements
$\phi(T_j)$ are linear combinations of the generators $X_1,\ldots,X_r$ of
$\Rad(kE)$. We call such embeddings
\emph{homogeneously embedded $s$-shifted subgroups}. As we shall see,
homogeneously embedded $s$-shifted subgroups give rise to closed
immersions $\mathbb{P}^{s-1}\hookrightarrow\mathbb{P}^{r-1}$ that are of
interest in the study of vector bundles on projective space, e.g.,
embeddings of lines into $\mathbb{P}^{r-1}$.

Following the notation of the previous section, we continue to let $V$
denote the subspace of $kE$ spanned by $X_1,\ldots,X_r$, and we let
$U$ be the subspace of $kE'$ spanned by $T_1,\ldots,T_s$.
By definition, the homogeneous embedding $\phi\colon kE'\hookrightarrow kE$
is given by a linear embedding $U\hookrightarrow V$, which is represented
by an $r\times s$ matrix $\mathbf{A}=(a_{ij})$. Specifically, we have
$\phi(T_j)=\sum_{i=1}^r a_{ij}X_i$. Taking $k$-linear duals, the matrix
$\mathbf{A}^t$ induces a surjective linear map $V^\#\rightarrow U^\#$.
Letting $Z_1,\ldots,Z_s$ denote the dual elements in $U^\#$ that correspond
to $T_1,\ldots,T_s$, respectively, $\mathbf{A}^t$ then gives rise to a
surjective graded homomorphism of $k$-algebras
\[
\xymatrix{
\phi^\#\colon k[Y_1,\ldots,Y_r] \ar@{>}[r] & k[Z_1,\ldots,Z_s].
}
\]
Specifically, we have $\phi^\#(Y_i)=\sum_{j=1}^s a_{ij}Z_j$.
Finally, applying the functor $\Proj$, the graded homomorphism $\phi^\#$
induces the desired closed immersion
$f\colon\mathbb{P}^{s-1}\hookrightarrow\mathbb{P}^{r-1}$.

For a finite dimensional $kE$-module $M$, we now wish to compare the
coherent sheaves $f^*(\mathcal{F}_i(M))$ and $\mathcal{F}_i(M\res{kE'})$
on $\mathbb{P}^{s-1}$, where by abuse of notation, $\mathcal{F}_i$ denotes
both functors
$\fgmod(kE)\rightarrow\coh(\mathbb{P}^{r-1})$ and
$\fgmod(kE')\rightarrow\coh(\mathbb{P}^{s-1})$, respectively.

\begin{proposition}
Let $M$ be any finite dimensional $kE$-module. Let $\mathcal{E}^{\bullet}$
denote the sequence of coherent sheaves
\[
\xymatrix @C=0.5in{
\cdots \ar@{>}[r] & \widetilde{M}(n-1) \ar@{>}[r]^{\ \ \theta_M} &
\widetilde{M}(n) \ar@{>}[r]^{\theta_M\ \ \ } & \widetilde{M}(n+1)
\ar@{>}[r] & \cdots
}
\]
on $\mathbb{P}^{r-1}$ and let $\mathcal{E}\res{kE'}^\bullet$ denote
the sequence of coherent sheaves
\[
\xymatrix @C=0.5in{
\cdots \ar@{>}[r] & \widetilde{M\res{kE'}}(n-1)
\ar@{>}[r]^{\ \ \theta_{M\res{kE'}}} & \widetilde{M\res{kE'}}(n)
\ar@{>}[r]^{\theta_{M\res{kE'}}\ \ \ } &
\widetilde{M\res{kE'}}(n+1) \ar@{>}[r] & \cdots
}
\]
on $\mathbb{P}^{s-1}$. (Observe that these are not chain complexes
unless $p=2$.) Then $f^*\mathcal{E}^{\bullet}$ is naturally isomorphic
to $\mathcal{E}\res{kE'}^\bullet$ in the functor category
$\Fun(\mathbb{Z},\coh(\mathbb{P}^{s-1}))$.
\end{proposition}

\begin{proof}
For $n\in\mathbb{Z}$, note that
$\widetilde{M}(n)=M\otimes_k\mathcal{O}_{\mathbb{P}^{r-1}}(n)$ is
isomorphic to $\mathcal{O}_{\mathbb{P}^{r-1}}(n)^{\oplus d}$, where
$d=\dim_k M$. In this way, $\theta_M=\sum_i X_i\otimes Y_i$ can be viewed
as the $d\times d$ matrix $\sum_i Y_iX_i$ with entries in
$k[Y_1,\ldots,Y_r]$. It follows that $f^*\theta_M$ is the matrix map
\[
\xymatrix @C=0.9in{
\mathcal{O}_{\mathbb{P}^{s-1}}(n)^{\oplus d}
=f^*(\mathcal{O}_{\mathbb{P}^{r-1}}(n)^{\oplus d})
\ar@{>}[r]^{\sum_i \phi^\#(Y_i)X_i\ \ \ \ \ \ } &
f^*(\mathcal{O}_{\mathbb{P}^{r-1}}(n+1)^{\oplus d})
=\mathcal{O}_{\mathbb{P}^{s-1}}(n+1)^{\oplus d}.
}
\]
On the other hand,
$\widetilde{M\res{kE'}}(n)=
M\res{kE'}\otimes_k\:\mathcal{O}_{\mathbb{P}^{s-1}}(n)$ can be
identified with $\mathcal{O}_{\mathbb{P}^{s-1}}(n)^{\oplus d}$,
hence we may view $\theta_{M\res{kE'}}=\sum_j T_j\otimes Z_j$
as a $d\times d$ matrix with entries in $k[Z_1,\ldots,Z_s]$.
But as a $k$-linear endomorphism of $M$, $T_j$ acts via the embedding
$\phi\colon kE'\rightarrow kE$. In other words,
$\theta_{M\res{kE'}}$ acts via the matrix
\[
\sum_j Z_j\phi(T_j)=\sum_{i,j} Z_j(a_{ij}X_i)
=\sum_{i,j} (a_{ij}Z_j)X_i=\sum_i \phi^\#(Y_i)X_i.
\]
Taking the vertical arrows in
\[
\xymatrix @C=0.8in{
\mathcal{O}_{\mathbb{P}^{s-1}}(n)^{\oplus d}
\ar@{>}[r]^{\hspace{-0.15in}\sum_i \phi^\#(Y_i)X_i} \ar@{>}[d] &
\mathcal{O}_{\mathbb{P}^{s-1}}(n+1)^{\oplus d} \ar@{>}[d] \\
\mathcal{O}_{\mathbb{P}^{s-1}}(n)^{\oplus d}
\ar@{>}[r]^{\hspace{-0.15in}\sum_j Z_j\phi(T_j)} &
\mathcal{O}_{\mathbb{P}^{s-1}}(n+1)^{\oplus d}
}
\]
to be the $d\times d$ identity matrix induces the required isomorphism
$f^*\mathcal{E}^\bullet\xrightarrow{\sim}\mathcal{E}\res{kE'}^\bullet$.
\end{proof}

\begin{corollary}\label{restrictionpullback}
If $M$ is any finite dimensional $kE$-module, then
\[
\mathcal{F}_i(M\res{kE'})\cong
\frac{\Ker(f^*\theta_M)\cap\Img(f^*\theta_M^{i-1})}
{\Ker(f^*\theta_M)\cap\Img(f^*\theta_M^i)}.
\]
\end{corollary}

Before giving our next result, we first recall an apparently standard
lemma whose proof we provide for completeness.

\begin{lemma}\label{leftexact}
Let $f\colon X\rightarrow Y$ be any morphism of locally ringed spaces. If
\[
\xymatrix{
0 \ar@{>}[r] & \mathcal{G}' \ar@{>}[r] & \mathcal{G} \ar@{>}[r] &
\mathcal{G}'' \ar@{>}[r] & 0
}
\]
is a short exact sequence of $\mathcal{O}_Y$-modules with $\mathcal{G}''$
locally free, then
\[
\xymatrix{
0 \ar@{>}[r] & f^*\mathcal{G}' \ar@{>}[r] & f^*\mathcal{G} \ar@{>}[r] &
f^*\mathcal{G}'' \ar@{>}[r] & 0
}
\]
is a short exact sequence of $\mathcal{O}_X$-modules.
\end{lemma}

\begin{proof}
The functor $f^*$ is right exact, so we have an exact sequence of
$\mathcal{O}_X$-modules
\[
\xymatrix{
L_1f^*\mathcal{G}'' \ar@{>}[r] & f^*\mathcal{G}' \ar@{>}[r] &
f^*\mathcal{G} \ar@{>}[r] & f^*\mathcal{G}'' \ar@{>}[r] & 0.
}
\]
Note that the leftmost term is
\[
L_1((-\otimes_{f^{-1}\mathcal{O}_Y}\mathcal{O}_X)\circ f^{-1})
(\mathcal{G}'')=
L_1(-\otimes_{f^{-1}\mathcal{O}_Y}\mathcal{O}_X)(f^{-1}\mathcal{G}'')
=\Tor^{f^{-1}\mathcal{O}_Y}_1(f^{-1}\mathcal{G}'',\mathcal{O}_X),
\]
where the first equality holds since $f^{-1}$ is exact. Now let
$x\in X$ and consider the localisa-
tion $(L_1f^*\mathcal{G}'')_x$.
Because $\Tor$ commutes with localisation, this is equal to
\[
\Tor^{f^{-1}\mathcal{O}_Y}_1(f^{-1}\mathcal{G}'',\mathcal{O}_X)_x=
\Tor^{\mathcal{O}_{Y,f(x)}}_1(\mathcal{G}''_{f(x)},\mathcal{O}_{X,x}).
\]
Since $\mathcal{G}''$ is locally free, the right hand term is zero, hence
$(L_1f^*\mathcal{G}'')_x=0$. This being true for all $x\in X$, it follows
that $L_1f^*\mathcal{G}''=0$.
\end{proof}

Returning to the case where
$f\colon\mathbb{P}^{s-1}\rightarrow\mathbb{P}^{r-1}$ is the closed
immersion corresponding to the embedding $\phi\colon kE'\rightarrow kE$,
we are now in a position to prove the following.

\begin{theorem}\label{pullback}
If $M$ is a $kE$-module of constant Jordan type, then
$f^*\mathcal{F}_i(M)\cong\mathcal{F}_i(M\res{kE'})$ for all
$1\leq i\leq p$.
\end{theorem}

\begin{proof}
We first claim that we can identify $f^*(\Img\theta_M^i)$ and
$\Img(f^*\theta_M^i)$ as subobjects of $f^*\widetilde{M}$ for all
$1\leq i\leq p$.  Since $M$ has constant Jordan type, the morphism
$\theta^{i}_{M}\colon\widetilde{M}(-i)\rightarrow\widetilde{M}$
has locally free cokernel.  By Lemma \ref{leftexact}, applying
$f^*$ gives an exact sequence 
%
%By the right exactness of $f^*$, there is
%a short exact sequence
%\[
%\xymatrix @C=0.5in{
%f^*\widetilde{M}(-i) \ar@{>}[r]^{\ \ f^*\theta_M^i} & f^*\widetilde{M}
%\ar@{>}[r] & f^*(\Coker\theta_M^i) \ar@{>}[r] & 0
%}
%\]
%of $\mathcal{O}_{\mathbb{P}^{s-1}}$-modules. We may therefore choose
%$f^*(\Coker\theta_M^i)$ as the fixed isomorphism class representative of
%$\Coker(f^*\theta_M^i)$. Note that since $M$ has constant Jordan type,
%$\Coker\theta_M^i$ is locally free. By Lemma \ref{leftexact}, this means
%that
%
\[
\xymatrix{
0 \ar@{>}[r] & f^*(\Img\theta_M^i) \ar@{>}[r] & f^*\widetilde{M}
\ar@{>}[r] & f^*(\Coker\theta_M^i) \ar@{>}[r] & 0
}
\]
By the universal property of cokernels, there is a unique isomorphism $u$
making the following diagram commute.
%\[
%\xymatrix{
%0 \ar@{>}[r] & f^*(\Img\theta_M^i) \ar@{>}[r] & f^*\widetilde{M}
%\ar@{>}[r] \ar@{=}[d] & f^*(\Coker\theta_M^i) \ar@{>}[r] \ar@{=}[d] & 0 \\
%0 \ar@{>}[r] & \Img(f^*\theta_M^i) \ar@{>}[r] & f^*\widetilde{M} \ar@{>}[r]
%& \Coker(f^*\theta_M^i) \ar@{>}[r] & 0
%}
%\]
\[
\xymatrix @C=0.5in{
0 \ar@{>}[r] &
f^*(\Img\theta_M^i) \ar@{>}[r]^-{f^{\ast}\theta^i_M} &
f^*\widetilde{M} \ar@{>}[r]^-{q'} \ar@{>}[rd]_{q} &
f^*(\Coker\theta_M^i)  \ar@{>}[r] & 0 \\
& & & \Coker(f^*\theta_M^i) \ar@{>}[u]_u &
}
\]
Since $u$ is an isomorphism, we have $\Ker q' = \Ker q$, hence 
$f^*(\Img\theta_M^i)=\Img(f^*\theta_M^i)$ as subobjects of
$f^*\widetilde{M}$. Now consider the short exact sequence
\[
\xymatrix{
0 \ar@{>}[r] & \Ker\theta_M\cap\Img\theta_M^{i-1} \ar@{>}[r] &
\Img\theta_M^{i-1} \ar@{>}[r]^{\theta_M\ } & \Img\theta_M^i(1) \ar@{>}[r]
& 0
}
\]
of coherent sheaves on $\mathbb{P}^{r-1}$. Since $M$ has constant
Jordan type, $\Img\theta_M^i$ is locally free, so a similar argument
shows that
$f^*(\Ker\theta_M\cap\Img\theta_M^{i-1})=
\Ker(f^*\theta_M)\cap\Img(f^*\theta_M^{i-1})$ 
as subobjects of $f^*\widetilde{M}$ for all $i$.
%By Lemma \ref{leftexact},
%we conclude that
%\[
%\xymatrix @C=0.4in{
%0 \ar@{>}[r] & f^*(\Ker\theta_M\cap\Img\theta_M^{i-1}) \ar@{>}[r] &
%f^*(\Img\theta_M^{i-1}) \ar@{>}[r]^{f^*\theta_M\ } & f^*(\Img\theta_M^i)(1)
%\ar@{>}[r] & 0
%}
%\]
%is again exact. The diagram of short exact sequences
%\[
%\xymatrix @C=0.4in{
%0 \ar@{>}[r] & f^*(\Ker\theta_M\cap\Img\theta_M^{i-1}) \ar@{>}[r] &
%f^*(\Img\theta_M^{i-1}) \ar@{>}[r]^{f^*\theta_M\ \ } \ar@{=}[d] &
%f^*(\Img\theta_M^i)(1) \ar@{>}[r] \ar@{=}[d] & 0 \\
%0 \ar@{>}[r] & \Ker(f^*\theta_M)\cap\Img(f^*\theta_M^{i-1}) \ar@{>}[r] &
%\Img(f^*\theta_M^{i-1}) \ar@{>}[r]^{f^*\theta_M\ \ } & \Img(f^*\theta_M^i)(1)
%\ar@{>}[r] & 0
%}
%\]
%then allows us to identify 

Finally, consider the short exact sequence
\[
\xymatrix{
0 \ar@{>}[r] & \Ker\theta_M\cap\Img\theta_M^i \ar@{>}[r] &
\Ker\theta_M\cap\Img\theta_M^{i-1} \ar@{>}[r] & \mathcal{F}_i(M)
\ar@{>}[r] & 0
}
\]
defining $\mathcal{F}_i(M)$. Because $M$ has constant Jordan type,
Proposition 2.1 of \cite{Benson/Pevtsova:vb} tells us that
$\mathcal{F}_i(M)$ is locally free. Another application
of Lemma \ref{leftexact} then reveals that
\[
\xymatrix{
0 \ar@{>}[r] & f^*(\Ker\theta_M\cap\Img\theta_M^i) \ar@{>}[r] &
f^*(\Ker\theta_M\cap\Img\theta_M^{i-1}) \ar@{>}[r] & f^*\mathcal{F}_i(M)
\ar@{>}[r] & 0
}
\]
is exact. In light of Corollary \ref{restrictionpullback}, the
diagram of short exact sequences
\[
\xymatrix{
0 \ar@{>}[r] & f^*(\Ker\theta_M\cap\Img\theta_M^i) \ar@{>}[r]
\ar@{=}[d] & f^*(\Ker\theta_M\cap\Img\theta_M^{i-1}) \ar@{>}[r]
\ar@{=}[d] & f^*\mathcal{F}_i(M) \ar@{>}[r] & 0 \\
0 \ar@{>}[r] & \Ker(f^*\theta_M)\cap\Img(f^*\theta_M^i) \ar@{>}[r] &
\Ker(f^*\theta_M)\cap\Img(f^*\theta_M^{i-1}) \ar@{>}[r] &
\mathcal{F}_i(M\res{kE'}) \ar@{>}[r] & 0
}
\]
immediately implies that
$f^*\mathcal{F}_i(M)\cong\mathcal{F}_i(M\res{kE'})$.
\end{proof}

\begin{remark}
A similar statement to Theorem \ref{pullback} cannot hold for the
pushforward along $f$. In particular, $f_*\mathcal{F}_i(M\res{kE'})$ is
never isomorphic to $\mathcal{F}_i(M)$ unless the latter sheaf is zero or
$s=r$. This is because the pushforward of a sheaf along a closed
immersion that is not surjective is never globally supported, hence
cannot be locally free of non-zero rank.
\end{remark}

\section{More on the operator
\texorpdfstring{$\theta_M$}{thetaM} and vector bundles}

This section consists of extensions of the
preliminary results in \cite{Benson/Pevtsova:vb}. We give an explicit exposition of these key
ideas, as they will be used extensively throughout the sequel.
We first recall what is perhaps the most important fact about
locally free sheaves on projective space. It is traditionally
referenced as Exercise II.5.8 of \cite{Hartshorne:1977a}.

\begin{lemma}\label{rankcokernellemma}
Let $X$ be a reduced connected noetherian scheme and
$f\colon\mathcal{E}\rightarrow\mathcal{E}'$ a morphism of locally
free sheaves on $X$. Then the dimension of the fibre
\[
\xymatrix @C=0.3in{
f\otimes k(x)\colon\mathcal{E}\otimes_{\mathcal{O}_X} k(x) \ar@{>}[r] &
\mathcal{E}'\otimes_{\mathcal{O}_X} k(x)
}
\]
is independent of $x\in X$ if and only if $\,\Coker f$ is locally free.

If these conditions hold, then the coherent sheaf $\,\Img f$ is also
locally free.
\end{lemma}

\begin{proof}
The second statement follows from the short exact sequence
\[
\xymatrix{
0 \ar@{>}[r] & \Img f \ar@{>}[r] & \mathcal{E}' \ar@{>}[r] &
\Coker f \ar@{>}[r] & 0
}
\]
in which the map on the right is a surjection of locally free sheaves.
\end{proof}

\begin{notation}
Observe that if $M$ is a $kE$-module, then for each non-zero
$\alpha=(\lambda_1,\ldots,\lambda_r)$ in $\mathbb{A}^r(k)$, the
submodules $\Img(X_\alpha^i,M)$ and $\Ker(X_\alpha^i,M)$ are
uniquely determined by the class
$\overline{\alpha}=[\lambda_1:\ldots:\lambda_r]$ in $\mathbb{P}^{r-1}(k)$.
In what follows, we shall often find it convenient to use the closed point
$\overline{\alpha}\in\mathbb{P}^{r-1}(k)$ to parameterise the action of
the non-zero element $X_\alpha$ on $M$.
\end{notation}

The following two lemmas were instrumental in the proof of Proposition 2.1
of \cite{Benson/Pevtsova:vb}. We provide details here, because the
same reasoning will be used later when we examine the behaviour of vector
bundles with respect to submodules.

\begin{lemma}\label{firstlemma}
Let $M$ be a $kE$-module and suppose that the rank of $X_\alpha^i$
acting as a $k$-linear endomorphism of $M$ is independent of the choice
of $\,\overline{\alpha}\in\mathbb{P}^{r-1}(k)$ for some $i\geq 0$.
Then the coherent sheaf $\,\Img\theta_M^i$ is locally free. Moreover,
the fibre of the short exact sequence
\[
\xymatrix{
0 \ar@{>}[r] & \Img\theta_M^i \ar@{>}[r] & \widetilde{M} \ar@{>}[r] &
\Coker\theta_M^i \ar@{>}[r] & 0
}
\]
at a point $\,\overline{\alpha}\in\mathbb{P}^{r-1}(k)$ may be identified
with the natural short exact sequence
\[
\xymatrix{
0 \ar@{>}[r] & \Img(X_\alpha^i,M) \ar@{>}[r] & M \ar@{>}[r] &
\Coker(X_\alpha^i,M) \ar@{>}[r] & 0.
}
\]
\end{lemma}

\begin{proof}
Recall that the fibre of
$\widetilde{M}(-i)\xrightarrow{\theta_M^i}\widetilde{M}$ at
$\overline{\alpha}\in\mathbb{P}^{r-1}(k)$ is the map
$M\xrightarrow{X_\alpha^i}M$, the rank of which is independent of
$\overline{\alpha}$. By Lemma \ref{rankcokernellemma}, this shows that
the coherent sheaves $\Coker\theta_M^i$ and $\Img\theta_M^i$ are locally
free.

For the statements regarding fibres, note that because the tensor
product is right exact, the exact sequence
\[
\xymatrix{
\widetilde{M}(-i) \ar@{>}[r]^{\ \ \theta_M^i} &
\widetilde{M} \ar@{>}[r] & \Coker\theta_M^i \ar@{>}[r] & 0
}
\]
gives rise to an exact sequence of vector spaces
\[
\xymatrix{
M \ar@{>}[r]^{X_\alpha^i} & M \ar@{>}[r] &
\Coker\theta_M^i\otimes
k(\overline{\alpha}) \ar@{>}[r] & 0.
}
\]
We therefore identify the fibre of $\Coker\theta_M^i$ at
$\overline{\alpha}$ with $\Coker(X_\alpha^i,M)$. Taking the fibre of
\[
\xymatrix{
0 \ar@{>}[r] & \Img\theta_M^i \ar@{>}[r] & \widetilde{M} \ar@{>}[r] &
\Coker\theta_M^i \ar@{>}[r] & 0
}
\]
at $\overline{\alpha}\in\mathbb{P}^{r-1}(k)$ and placing it in the
top row of the diagram
\[
\xymatrix{
0 \ar@{>}[r] & \Img\theta_M^i\otimes k(\overline{\alpha}) \ar@{>}[r] &
M \ar@{>}[r] \ar@{=}[d] &
\Coker\theta_M^i\otimes k(\overline{\alpha}) \ar@{>}[r] \ar@{=}[d] &
0 \\
0 \ar@{>}[r] & \Img(X_\alpha^i,M) \ar@{>}[r] & M \ar@{>}[r] &
\Coker(X_\alpha^i,M) \ar@{>}[r] & 0
}
\]
then allows us to identify the fibre of $\Img\theta_M^i$ at
$\overline{\alpha}$ with $\Img(X_\alpha^i,M)$.
\end{proof}

\begin{lemma}\label{kernelcapimage}
If the ranks of $X_\alpha^i$ and $X_\alpha^{i+1}$ acting on $M$ are both
independent of the choice of $\,\overline{\alpha}\in\mathbb{P}^{r-1}(k)$
(so that $\,\Img\theta_M^i$ and $\,\Img\theta_M^{i+1}$ are locally free),
then $\,\Ker\theta_M\cap\Img\theta_M^i$ is
also locally free, and the fibre of the short exact sequence
\[
\xymatrix{
0 \ar@{>}[r] & \Ker\theta_M\cap\Img\theta_M^i \ar@{>}[r] &
\Img\theta_M^i \ar@{>}[r]^{\theta_M\ \ \ } &
\Img\theta_M^{i+1}(1) \ar@{>}[r] & 0
}
\]
at $\,\overline{\alpha}\in\mathbb{P}^{r-1}(k)$ is the short exact
sequence of vector spaces
\[
\xymatrix{
0 \ar@{>}[r] & \Ker(X_\alpha,M)\cap\Img(X_\alpha^i,M) \ar@{>}[r] &
\Img(X_\alpha^i,M) \ar@{>}[r]^{X_\alpha\ \ } &
\Img(X_\alpha^{i+1},M) \ar@{>}[r] & 0.
}
\]
\end{lemma}

\begin{proof}
Consider the diagram of short exact sequences of locally free sheaves
\[
\xymatrix{
0 \ar@{>}[r] & \Img\theta_M^i \ar@{>}[r] \ar@{>}[d] &
\widetilde{M} \ar@{>}[r] \ar@{>}[d]^{\theta_M} &
\Coker\theta_M^i \ar@{>}[r] \ar@{>}[d] & 0 \\
0 \ar@{>}[r] & \Img\theta_M^{i+1}(1) \ar@{>}[r] &
\widetilde{M}(1) \ar@{>}[r] & \Coker\theta_M^{i+1}(1) \ar@{>}[r] & 0
}
\]
whose fibre at $\overline{\alpha}\in\mathbb{P}^{r-1}(k)$ is
\[
\xymatrix{
0 \ar@{>}[r] & \Img(X_\alpha^i,M) \ar@{>}[r] \ar@{>}[d] &
M \ar@{>}[r] \ar@{>}[d]^{X_\alpha} &
\Coker(X_\alpha^i,M) \ar@{>}[r] \ar@{>}[d] & 0 \\
0 \ar@{>}[r] & \Img(X_\alpha^{i+1},M) \ar@{>}[r] &
M \ar@{>}[r] & \Coker(X_\alpha^{i+1},M) \ar@{>}[r] & 0.
}
\]
This immediately shows that the fibre of
$\theta_M\colon\Img\theta_M^i\rightarrow\Img\theta_M^{i+1}(1)$ at
$\overline{\alpha}$ is the induced map
$X_\alpha\colon\Img(X_\alpha^i,M)\rightarrow\Img(X_\alpha^{i+1},M)$.
The proof now follows as described in Proposition 2.1 of
\cite{Benson/Pevtsova:vb}.
\end{proof}

Now let $M$ be a $kE$-module and $N$ any submodule of $M$.
The inclusion $N\subseteq M$ induces an inclusion of
locally free sheaves $\widetilde{N}\subseteq\widetilde{M}$, and one
readily confirms that the diagrams
\[
\xymatrix{
\widetilde{N}(n) \ar@{>}[r]^{\theta_N\ \ \ } \ar@{^{(}->}[d] &
\widetilde{N}(n+1) \ar@{^{(}->}[d] \\
\widetilde{M}(n) \ar@{>}[r]^{\theta_M\ \ \ } & \widetilde{M}(n+1)
}
\]
commute. It follows that $\Img\theta_N^i\subseteq\Img\theta_M^i$
and $\Ker\theta_N^i\subseteq\Ker\theta_M^i$ for each $i\geq 0$.

The following proposition will be the essential step in showing that
$\mathcal{F}_i(N)=\mathcal{F}_i(M)$ in certain cases.

\begin{proposition}\label{fibresofinclusion}
If $M$ and $N$ both have constant $i$- and $(i+1)$-rank,
then the fibre of the natural inclusion
$\Ker\theta_N\cap\Img\theta_N^i\subseteq\Ker\theta_M\cap\Img\theta_M^i$
at a point $\overline{\alpha}\in\mathbb{P}^{r-1}(k)$ is the
inclusion of vector spaces
$\,\Ker(X_\alpha,N)\cap\Img(X_\alpha^i,N)\subseteq
\Ker(X_\alpha,M)\cap\Img(X_\alpha^i,M)$.
\end{proposition}

\begin{proof}
Since both $M$ and $N$ have constant $i$-rank, Lemma \ref{firstlemma}
tells us that the fibre of the diagram of short exact sequences
\[
\xymatrix{
0 \ar@{>}[r] & \Img\theta_N^i \ar@{>}[r] \ar@{^{(}->}[d] &
\widetilde{N} \ar@{>}[r] \ar@{^{(}->}[d] &
\Coker\theta_N^i \ar@{>}[r] \ar@{>}[d] & 0 \\
0 \ar@{>}[r] & \Img\theta_M^i \ar@{>}[r] &
\widetilde{M} \ar@{>}[r] & \Coker\theta_M^i \ar@{>}[r] & 0
}
\]
at a point $\overline{\alpha}\in\mathbb{P}^{r-1}(k)$ is
\[
\xymatrix{
0 \ar@{>}[r] & \Img(X_\alpha^i,N) \ar@{>}[r] \ar@{>}[d] &
N \ar@{>}[r] \ar@{>}[d] &
\Coker(X_\alpha^i,N) \ar@{>}[r] \ar@{>}[d] & 0 \\
0 \ar@{>}[r] & \Img(X_\alpha^i,M) \ar@{>}[r] &
M \ar@{>}[r] & \Coker(X_\alpha^i,M) \ar@{>}[r] & 0,
}
\]
where the middle map is the inclusion $N\subseteq M$. This implies
that the fibre of the inclusion $\Img\theta_N^i\subseteq\Img\theta_M^i$
at $\overline{\alpha}$ is the natural inclusion
$\Img(X_\alpha^i,N)\subseteq\Img(X_\alpha^i,M)$. The same argument also
shows that the fibre of $\Img\theta_N^{i+1}\subseteq\Img\theta_M^{i+1}$
is $\Img(X_\alpha^{i+1},N)\subseteq\Img(X_\alpha^{i+1},M)$.

Now consider the diagram of short exact sequences
\[
\xymatrix{
0 \ar@{>}[r] &
\Ker\theta_N\cap\Img\theta_N^i \ar@{>}[r] \ar@{^{(}->}[d] &
\Img\theta_N^i \ar@{>}[r]^{\theta_N\ \ \ } \ar@{^{(}->}[d] &
\Img\theta_N^{i+1}(1) \ar@{>}[r] \ar@{^{(}->}[d] & 0 \\
0 \ar@{>}[r] & \Ker\theta_M\cap\Img\theta_M^i \ar@{>}[r] &
\Img\theta_M^i \ar@{>}[r]^{\theta_M\ \ \ } &
\Img\theta_M^{i+1}(1) \ar@{>}[r] & 0
}
\]
whose terms are all locally free, and whose fibre at a point
$\overline{\alpha}\in\mathbb{P}^{r-1}(k)$ is
\[
\xymatrix{
0 \ar@{>}[r] &
\Ker(X_\alpha,N)\cap\Img(X_\alpha^i,N) \ar@{>}[r] \ar@{>}[d] &
\Img(X_\alpha^i,N) \ar@{>}[r]^{X_\alpha\ \ } \ar@{^{(}->}[d] &
\Img(X_\alpha^{i+1},N) \ar@{>}[r] \ar@{^{(}->}[d] & 0 \\
0 \ar@{>}[r] & \Ker(X_\alpha,M)\cap\Img(X_\alpha^i,M) \ar@{>}[r] &
\Img(X_\alpha^i,M) \ar@{>}[r]^{X_\alpha\ \ } &
\Img(X_\alpha^{i+1},M) \ar@{>}[r] & 0.
}
\]
In light of the above remarks, the right two maps are the natural
inclusions $\Img(X_\alpha^i,N)\subseteq\Img(X_\alpha^i,M)$ and
$\Img(X_\alpha^{i+1},N)\subseteq\Img(X_\alpha^{i+1},M)$, respectively.
This forces the map on the left to be the inclusion
$\Ker(X_\alpha,N)\cap\Img(X_\alpha^i,N)\subseteq
\Ker(X_\alpha,M)\cap\Img(X_\alpha^i,M)$.
\end{proof}

\section{The equal images property and vector bundles}

In this section we recall the notion of the equal images property for
$kE$-modules. For such a module $M$, we introduce an inductive procedure
for computing the vector bundles $\mathcal{F}_i(M)$. As was the case in
Section \ref{pullbacksection}, there is a close relationship between the
structure of such modules and the geometry of $\mathbb{P}^{r-1}(k)$.
Later in the section, we use this procedure to compute the vector bundles
for so-called `$W$-modules' in the case $r=2$. The following definition
first appeared in \cite{Carlson/Friedlander/Suslin:2011a}.

\begin{definition}\label{eipdefinition}
A $kE$-module $M$ has the \emph{equal images property} if the image of
$X_\alpha$ acting on $M$ is independent of the choice of
$\overline{\alpha}\in\mathbb{P}^{r-1}(k)$.
\end{definition}

A useful characterisation of the equal images property is the following,
which appeared as Proposition 2.5 of \cite{Carlson/Friedlander/Suslin:2011a}.

\begin{proposition}\label{eipradicals}
A $kE$-module $M$ has the equal images property if and only if the
image of $X_\alpha$ acting on $M$ is equal to $\Rad(M)$ for all
$\overline{\alpha}\in\mathbb{P}^{r-1}(k)$.
\end{proposition}

We remark that the equal images property is rather strong. In particular,
if $M$ has the equal images property, then $M$ has constant Jordan type,
although the converse does not necessarily hold.
(See Proposition 2.8 of \cite{Carlson/Friedlander/Suslin:2011a}.)
We provide the following brief summary of the salient points in Section 2
of \cite{Carlson/Friedlander/Suslin:2011a}.

\begin{proposition}\label{eipproperties}
The class of $kE$-modules with the equal images property is closed under
taking direct sums, quotients and radicals.
\end{proposition}

The following is our main result regarding the inductive nature of
the functors $\mathcal{F}_i$ evaluated at modules having the equal images
property.

\begin{theorem}\label{eipradicalvb}
If $M$ is a $kE$-module with the equal images property, then for all
$\,0\leq j<i\leq p$ we have
$\mathcal{F}_i(M)\cong\mathcal{F}_{i-j}(\Rad^j(M))$.
\end{theorem}

\begin{proof}
By the construction of $\theta_M$, the map
$\theta_M^j\colon\widetilde{M}(-i)\rightarrow\widetilde{M}(-i+j)$
factors through the vector bundle $\widetilde{\Rad^j(M)}(-i+j)$. Let
$\widehat{\theta_M^j}$ denote the induced map
\[
\xymatrix{
\widetilde{M}(-i) \ar@{>}[r] & \widetilde{\Rad^j(M)}(-i+j).
}
\]
We then have a commutative diagram
\[
\xymatrix @C=0.7in{
\widetilde{M}(-i) \ar@{>}[r]^{\widehat{\theta_M^j}\ \ \ \ \ \ \ }
\ar@{=}[d] &
\widetilde{\Rad^j(M)}(-i+j) \ar@{>}[r]^{\ \ \ \ \ \theta_{\Rad^j(M)}^{i-j}}
\ar@{^{(}->}[d] &
\widetilde{\Rad^j(M)} \ar@{^{(}->}[d] \\
\widetilde{M}(-i) \ar@{>}[r]^{\theta_M^j\ \ } &
\widetilde{M}(-i+j) \ar@{>}[r]^{\ \ \ \ \ \theta_M^{i-j}} &
\widetilde{M}
}
\]
where the right two vertical arrows are those induced by the inclusion
$\Rad^j(M)\subseteq M$.  We claim that $\widehat{\theta_M^j}$ is surjective,
from which it will follow that the image of $\theta_{\Rad^j(M)}^{i-j}$
equals that of $\theta_M^i$. To see this, note that because $\theta_M^j$
is a map of vector bundles and the middle vertical arrow is an injection,
$\widehat{\theta_M^j}$ is also a map of vector bundles.
This also allows one to conclude that the fibre of $\widehat{\theta_M^j}$ at a
point $\overline{\alpha}\in\mathbb{P}^{r-1}(k)$ is the linear
map $X_\alpha^j\colon M\rightarrow\Rad^j(M)$. Since $M$ has the
equal images property, Propositions \ref{eipradicals} and \ref{eipproperties}
imply that this map is always surjective. In other words,
$\widehat{\theta_M^j}$ is a map of vector bundles that
is surjective on fibres, hence surjective.

We next claim that there is an equality
\[
\Ker\theta_{\Rad^j(M)}\cap\Img\theta_{\Rad^j(M)}^{i-j}=
\Ker\theta_M\cap\Img\theta_{\Rad^j(M)}^{i-j}
\]
as subsheaves of $\widetilde{M}$. There is an obvious rightwards inclusion
induced by the inclusion of modules $\Rad^j(M)\subseteq M$. The reverse containment
follows from the fact that the left hand side is the kernel of
$\theta_{\Rad^j(M)}\colon\Img\theta_{\Rad^j(M)}^{i-j}\rightarrow
\Img\theta_{\Rad^j(M)}^{i-j+1}(1)$ which, after precomposing with the
inclusion $\Ker\theta_M\cap\Img\theta_{\Rad^j(M)}^{i-j}\hookrightarrow
\Img\theta_{\Rad^j(M)}^{i-j}$, yields the zero map.

Putting this all together, we therefore have
\[
\begin{aligned}
\mathcal{F}_i(M) &=
\frac{\Ker\theta_M\cap\Img\theta_M^{i-1}}
{\Ker\theta_M\cap\Img\theta_M^i}\cong
\frac{\Ker\theta_M\cap\Img\theta_{\Rad^j(M)}^{i-j-1}}
{\Ker\theta_M\cap\Img\theta_{\Rad^j(M)}^{i-j}}=
\frac{\Ker\theta_{\Rad^j(M)}\cap\Img\theta_{\Rad^j(M)}^{i-j-1}}
{\Ker\theta_{\Rad^j(M)}\cap\Img\theta_{\Rad^j(M)}^{i-j}} \\
&= \mathcal{F}_{i-j}(\Rad^j(M))
\end{aligned}
\]
as required.
\end{proof}

Before giving an application of Theorem \ref{eipradicalvb}, we
recall the theory of Chern classes and how they interact with
the functors $\mathcal{F}_i$. Note that the Chow ring
$A^*(\mathbb{P}^{r-1}(k))$ of projective space is isomorphic to the
truncated polynomial ring $\mathbb{Z}[h]/h^r$. If $\mathcal{F}$ is
a vector bundle on $\mathbb{P}^{r-1}(k)$, then the \emph{Chern class}
of $\mathcal{F}$ is the well defined polynomial class
\[
c(\mathcal{F})=1+c_1(\mathcal{F})h+\cdots+c_{r-1}(\mathcal{F})h^{r-1}\in
A^*(\mathbb{P}^{r-1}(k))
\]
characterised by the following properties. (See Chapters 3 and 4 of
\cite{Fulton:1984a} for details.)
\begin{itemize}
\item[(1)] $c_i(\mathcal{F})=0$ for all $i\geq\rank(\mathcal{F})$.
\item[(2)] If $0\rightarrow\mathcal{F}_1\rightarrow\mathcal{F}_2\rightarrow
\mathcal{F}_3\rightarrow 0$ is a short exact sequence of vector bundles
on $\mathbb{P}^{r-1}(k)$, then
$c(\mathcal{F}_2)=c(\mathcal{F}_1)c(\mathcal{F}_3)$.
\item[(3)] $c(\mathcal{O}(n))=1+nh$ for all $n\in\mathbb{Z}$.
\end{itemize}
The integers $c_i(\mathcal{F})$ are called the \emph{Chern numbers} of
$\mathcal{F}$.
We record the formula for Chern numbers of twists, which follows from
Example 3.2.2 of \cite{Fulton:1984a}.

\begin{lemma}\label{chernnumbersoftwists}
If $\mathcal{F}$ is a vector bundle on $\mathbb{P}^{r-1}(k)$, then for all
$n\in\mathbb{Z}$, the $i$th Chern number of $\mathcal{F}(n)$ is given by
\[
c_i(\mathcal{F}(n))=\sum_{j=0}^i n^j
\binom{\rank(\mathcal{F})-i+j}{j}c_{i-j}(\mathcal{F}).
\]
\end{lemma}

We shall combine this fact with the following result, which follows from
Lemmas 2.2 and 2.3 of \cite{Benson/Pevtsova:vb}.

\begin{proposition}\label{filtrationproposition}
If $M$ is a finitely generated $kE$-module, then $\widetilde{M}$ has
a filtration whose filtered quotients are $\mathcal{F}_i(M)(j)$ for all
$0\leq j<i\leq p$.
\end{proposition}

\section{Application: Vector bundles for \texorpdfstring{$W$}{W}-modules}
\label{Wmodulesection}

In this section we restrict our attention to the case in which $E$ has
rank two and look at $W$-modules for $kE$. Such modules were first
introduced in \cite{Carlson/Friedlander/Suslin:2011a} and shown there to
play an important role in the theory of modules having the equal images
property. Our goal here is to use the results of the previous section to
compute the vector bundle $\mathcal{F}_i(M)$ for any $W$-module $M$.
Again, we emphasise that throughout the section we shall require the
rank $E$ to equal two, that is, $E\cong\mathbb{Z}/p\times\mathbb{Z}/p$.

\begin{definition}
Let $n$ and $d$ be positive integers such that $1\leq d\leq n$ and $d\leq p$.
If $V$ is the free $kE$-module of rank $n$ with generators $v_1,\ldots,v_n$,
we define $W_{n,d}$ to be the quotient $V/U$, where $U$ is the
$kE$-submodule of $V$ generated by the elements
\[
X_1v_1,\ \ \ X_2v_n,\ \ \ X_1^dv_i\ \ \text{for $1\leq i\leq n$,}\ \ \ 
X_2v_i-X_1v_{i+1}\ \ \text{for $1\leq i\leq n-1$.}
\]
Any $kE$-module of the form $W_{n,d}$ is called a \emph{$W$-module}.
\end{definition}

It is convenient to picture the structure of a $W$-module by way of
certain diagrams. For example, if $p$ is any prime number and $n\geq 2$,
then the module $W_{n,2}$ can be represented by the diagram
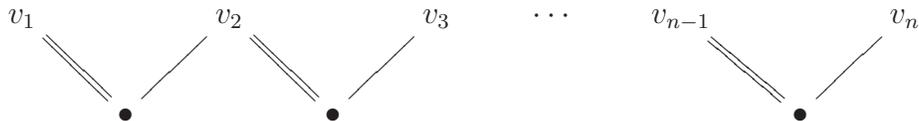
\[
\xymatrix{
v_1 \ar@{=}[dr] & & v_2 \ar@{-}[dl] \ar@{=}[dr]
& & v_3 \ar@{-}[dl]
& \cdots & v_{n-1} \ar@{=}[dr] & & v_n \ar@{-}[dl] \\
& \bullet & & \bullet & & & & \bullet &
}
\]
where each vertex represents a basis element of $W_{n,2}$. The generator
$X_1$ of $\Rad(kE)$ maps a vertex to the one lying below it, on the
opposite end of a single edge. Similarly, $X_2$ maps a vertex to
the the one lying below it, on the opposite end of a double edge.
Note that the Loewy length of a given module is indicated by the
number of rows in the corresponding diagram.
As another example, for $p\geq 3$, the module $W_{4,3}$ has diagram
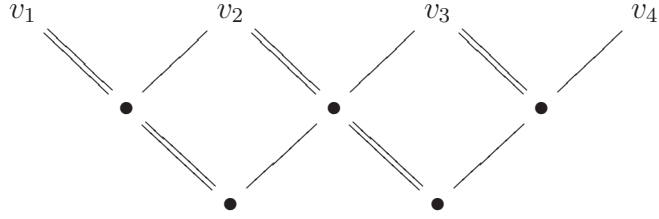
\[
\xymatrix{
v_1 \ar@{=}[dr] & & v_2 \ar@{-}[dl] \ar@{=}[dr]
& & v_3 \ar@{-}[dl] \ar@{=}[dr] & & v_4 \ar@{-}[dl] \\
& \bullet \ar@{=}[dr] & & \bullet \ar@{-}[dl] \ar@{=}[dr] & &
\bullet \ar@{-}[dl] & \\
& & \bullet & & \bullet & &
}
\]
Observe that both of the above diagrams roughly have a `W' shape, hence the
terminology `$W$-module'.

The following appeared as Proposition 3.3 of
\cite{Carlson/Friedlander/Suslin:2011a}.

\begin{proposition}\label{wmoduleshaveeip}
If $1\leq d\leq n$ and $d\leq p$, then the $kE$-module $W_{n,d}$ has
the equal images property.
\end{proposition}

As an immediate corollary, one obtains the following, which also appeared
in \cite{Carlson/Friedlander/Suslin:2011a}.

\begin{corollary}\label{jordantypesofwmodules}
If $1\leq d\leq n$ and $d\leq p$, then $W_{n,d}$ has constant Jordan type
\[
[d]^{n-d+1}[d-1]\ldots[1].
\]
\end{corollary}

\begin{proof}
The fact that $W_{n,d}$ has constant Jordan type follows from
Proposition \ref{wmoduleshaveeip}. Calculating its Jordan type is then
accomplished by calculating the Jordan type of $X_1$ on $W_{n,d}$, using
the corresponding module diagram.
\end{proof}

The surprising fact about $W$-modules is not that they have the equal
images property, but that they are, in some sense, nice models for all
$kE$-modules having the equal images property. This is made precise
by the next result, which appeared as Theorem 5.4 of
\cite{Carlson/Friedlander/Suslin:2011a}.

\begin{proposition}
If $M$ is a $kE$-module having the equal images property of radical length
$d$, then there exists an integer $n\geq d$ and a surjective module
homomorphism $W_{n,d}\rightarrow M$.
\end{proposition}

Motivated by the central role $W$-modules play in the theory of
$kE$-modules having the equal images property, we now calculate the
vector bundle $\mathcal{F}_i(W_{n,d})$ for each $1\leq i\leq d$.
Before doing so, we should point out that the kernel bundle
$\Ker\theta_{W_{n,d}}$ was calculated in Proposition 6.4 of
\cite{Carlson/Friedlander/Suslin:2011a}. Given that the vector bundles
$\mathcal{F}_i(W_{n,d})(j)$ form a filtration of the kernel bundle, our
calculation may be viewed as a refinement of this earlier work.

We recall that Grothendieck \cite{Grothendieck:1957a} has classified the
vector bundles on $\mathbb{P}^1(k)$. In particular, every such bundle is a
direct sum of line bundles
\[
\mathcal{O}_{\mathbb{P}^1}(n_1)\oplus\cdots\oplus
\mathcal{O}_{\mathbb{P}^1}(n_t),
\]
where the integers $n_1,\ldots,n_t$ are uniquely determined up to reordering.
Given this classification, we now present the main theorem of this section.

\begin{theorem}\label{wmodulevectorbundles}
If $1\leq d\leq n$ and $d\leq p$, then
$\mathcal{F}_i(W_{n,d})\cong\mathcal{O}_{\mathbb{P}^1}(-n+i)$ for
$1\leq i\leq d-1$, and
$\mathcal{F}_d(W_{n,d})\cong\mathcal{O}_{\mathbb{P}^1}^{\oplus(n-d+1)}$.
\end{theorem}

\begin{remark}
The ranks of these vector bundles are given by the exponents in the
statement of Corollary \ref{jordantypesofwmodules}.
\end{remark}

\begin{proof}[Proof of Theorem \ref{wmodulevectorbundles}]
We proceed by induction on $d$, the case $d=1$ being trivial. So suppose
$d>1$. Since the trivial bundle $\widetilde{W_{n,d}}$ has a filtration with
filtered quotients $\mathcal{F}_i(M)(j)$ for $0\leq j<i\leq d$ by
Proposition \ref{filtrationproposition}, we have
\[
1=c(\widetilde{W_{n,d}})=\prod_{0\leq j< i\leq d}
c\left(\mathcal{F}_i(W_{n,d})(j)\right).
\]
Comparing the first Chern numbers using Lemma \ref{chernnumbersoftwists}
gives us
\begin{equation}\label{firstchern}
0=\sum_{i=1}^d\left(ic_1(\mathcal{F}_i(W_{n,d}))
+\textstyle{\frac{1}{2}}i(i-1)\right).
\end{equation}
Note that $\Rad(W_{n,d})$ is also a $W$-module, being isomorphic to
$W_{n-1,d-1}$. By induction and Theorem \ref{eipradicalvb} we therefore
have
\[
\mathcal{F}_i(W_{n,d})\cong\mathcal{F}_{i-1}(\Rad(W_{n,d}))
\cong\mathcal{F}_{i-1}(W_{n-1,d-1})\cong\mathcal{O}_{\mathbb{P}^1}(-n+i)
\]
for $2\leq i\leq d-1$ and
\[
\mathcal{F}_d(W_{n,d})\cong\mathcal{F}_{d-1}(\Rad(W_{n,d}))
\cong\mathcal{F}_{d-1}(W_{n-1,d-1})\cong
\mathcal{O}_{\mathbb{P}^1}^{\oplus(n-d+1)}.
\]
Substituting this into (\ref{firstchern}) and simplifying yields
\[
0=c_1(\mathcal{F}_1(W_{n,d}))+n-1
\]
so that $\mathcal{F}_1(W_{n,d})\cong\mathcal{O}_{\mathbb{P}^1}(-n+1)$.
This completes the proof.
\end{proof}

As a corollary, we obtain the following bit of folklore regarding the
functor
\[
\xymatrix{
\mathcal{F}_1\colon\cJt(kE) \ar@{>}[r] & \vect(\mathbb{P}^1(k)).
}
\]
The fact that $\mathcal{F}_1$ is essentially surjective certainly
follows from Theorem 1.1 of \cite{Benson/Pevtsova:vb}, but the above
calculation allows us to further deduce that every vector bundle on
$\mathbb{P}^1(k)$ is of the form $\mathcal{F}_1(M)$ where $M$ has Loewy
length at most two. This result was related to the first author by
Dave Benson whilst the former was a student of the latter.
We require a quick lemma relating vector bundles to $k$-linear duals of
modules, which appeared as Theorem 3.6 of \cite{Benson/Pevtsova:vb}.

\begin{lemma}\label{vectorbundlesofduals}
If $M$ is a $kE$-module in any rank $r$, then
$\mathcal{F}_i(M^\#)\cong\mathcal{F}_i(M)^\vee(-i+1)$.
\end{lemma}

We are now in a position to prove the folklore indicated above.

\begin{corollary}
Every line bundle on $\mathbb{P}^1(k)$ is isomorphic to
$\mathcal{F}_1(k)$, $\mathcal{F}_1(W_{n,2})$ or
$\mathcal{F}_1(W_{n,2}^\#)$ for some $n$.
\end{corollary}

\begin{proof}
Theorem \ref{wmodulevectorbundles} tells us that
$\mathcal{F}_1(k)\cong\mathcal{F}_1(W_{1,1})\cong
\mathcal{O}_{\mathbb{P}^1}$
and
$\mathcal{F}_1(W_{n,2})\cong\mathcal{O}_{\mathbb{P}^1}(-n+1)$ for all
$n\geq 2$.
To realise line bundles with positive Chern numbers, we
use Lemma \ref{vectorbundlesofduals} to obtain
\[
\mathcal{F}_1(W_{n,2}^\#)\cong\mathcal{F}_1(W_{n,2})^\vee\cong
\mathcal{O}_{\mathbb{P}^1}(-n+1)^\vee\cong\mathcal{O}_{\mathbb{P}^1}(n-1)\qquad
\text{for all $n\geq 2$}.\qedhere
\]
\end{proof}

\section{Recollections about the generic kernel filtration}

The category of $kE$-modules of constant Jordan type is wild, even in the
case $r=2$. It was shown by Benson that even the category of such modules
having Loewy length three is wild.
(See Section 4.5 of \cite{Benson:book} for details.) 
On the other hand, the vector bundles on $\mathbb{P}^1(k)$ are rather well
behaved, which leaves one to wonder whether or not there is some sort of
structural invariant of a $kE$-module $M$ that completely determines
$\mathcal{F}_i(M)$, preferably one that is easy to understand.

For $r=2$, it turns out that there exists a filtration of $M$ that does
allow us to compute $\mathcal{F}_i(M)$ on certain, generally much smaller
subquotients of $M$. This filtration is related to the generic kernel of
a $k(\mathbb{Z}/p)^2$-module, which was introduced by 
Carlson, Friedlander and Suslin \cite{Carlson/Friedlander/Suslin:2011a}.
The following definition applies to $kE$-modules in arbitrary rank $r$.

\begin{definition}
Let $j\in\mathbb{N}$. A $kE$-module $M$ has
\emph{constant $j$-rank} if the rank of $X_\alpha^j$ acting on $M$ is
independent of $\overline{\alpha}\in\mathbb{P}^{r-1}(k)$. If $M$ has
constant 1-rank, then we simply say that $M$ has \emph{constant rank}.
\end{definition}

\begin{remark}
It is easy to see that a $kE$-module has constant Jordan type if and only
if it has constant $j$-rank for all $1\leq j\leq p$.
\end{remark}

Throughout the remainder of this section, we again restrict our attention
to the case in which $r=2$.

\begin{definition}
Let $M$ be a $kE$-module. For any cofinite subset $S$ of $\mathbb{P}^1(k)$,
consider the submodule
\[
\leftind{S}{M}=\sum_{\overline{\alpha}\in S}\Ker(X_\alpha,M)
\]
of $M$. The \emph{generic kernel} of $M$ is then defined to be the
submodule
\[
\KM=\bigcap_{\text{$S\subseteq\mathbb{P}^1(k)$ cofinite}}\leftind{S}{M}.
\]
\end{definition}

The following are the main results in Section 7 of
\cite{Carlson/Friedlander/Suslin:2011a} concerning the generic kernel.

\begin{lemma}\label{generickernelproperties}
Let $M$ be a $kE$-module.
\begin{itemize}
\item[(1)]
The generic kernel $\KM$ has the equal images property. Moreover,
if $N$ is any submodule of $M$ having the equal images property,
then $\KM$ contains $N$.
\item[(2)]
If $M$ has constant rank and $\,\overline{\alpha}\in\mathbb{P}^1(k)$,
then $\KM$ contains the kernel of the action of $X_\alpha$ on $M$.
\end{itemize}
\end{lemma}

Now consider the filtration
\[
0=J^p\KM\subseteq\cdots\subseteq J\KM\subseteq\KM
\subseteq J^{-1}\KM\subseteq \cdots\subseteq J^{-p+1}\KM
=M
\]
of $M$. (For $j\in\mathbb{N}$, $J^{-j}\KM$ denotes the set of elements
$m\in M$ for which $J^jm\subseteq\KM$.) We call the above filtration the
\emph{generic kernel filtration} of $M$. Our goal is to show that the
functors $\mathcal{F}_i$ behave well with respect to the generic kernel
filtration in the sense that $\mathcal{F}_i(M)$ may be computed on the
subquotient $J^{-i}\KM/J^{i+1}\KM$.
The following lemma appeared as Proposition 2.7 of \cite{Baland:2012a}.

\begin{lemma}\label{inversesofgenerickernels}
If $\,M$ is a $kE$-module of constant rank and
$\,\overline{\alpha}\in\mathbb{P}^1(k)$, then for all $j\geq 0$ we have
$X_\alpha^{-j}\KM=J^{-j}\KM$.
\end{lemma}

An easy consequence of this is the following, which appeared as
Lemma 5.24 of \cite{Baland:thesis}.

\begin{lemma}\label{kernelcapimagelemma}
If $M$ is a $kE$-module of constant rank, then for all
$\,\overline{\alpha}\in\mathbb{P}^1(k)$ and all $i\leq j$ we have
\[
\Ker(X_\alpha,J^{-j}\KM)\cap\Img(X_\alpha^i,J^{-j}\KM)=
\Ker(X_\alpha,M)\cap\Img(X_\alpha^i,M).
\]
\end{lemma}

\begin{proof}
The rightwards containment is clear, so let
$m\in\Ker(X_\alpha,M)\cap\Img(X_\alpha^i,M)$. Observe that
$\Ker(X_\alpha,M)\subseteq\KM$ by Lemma \ref{generickernelproperties} (2).
It follows that
\[
\Ker(X_\alpha,J^{-j}\KM)\subseteq\Ker(X_\alpha,M)\subseteq
\Ker(X_\alpha,\KM)\subseteq\Ker(X_\alpha,J^{-j}\KM),
\]
whence equality holds throughout. In particular, we have
$m\in\Ker(X_\alpha,J^{-j}\KM)$. Also, there exists $m'\in M$ such that
$X_\alpha^im'=m$. Since $m\in\KM$, Lemma
\ref{inversesofgenerickernels} implies that
$m'\in X_\alpha^{-i}\KM=J^{-i}\KM\subseteq J^{-j}\KM$, thus
$m\in\Img(X_\alpha^i,J^{-j}\KM)$.
\end{proof}

Although one may prove the results of the following section directly,
the presentation is made considerably more elegant via the following
duality statement related to the generic kernel filtration. It appeared as
Theorem 3.3 of \cite{Baland:2012a}.

\begin{lemma}\label{generickernelduality}
If $M$ is a $kE$-module of constant rank and $a,b\in\mathbb{Z}$ satisfy
$a\leq b$, then
\[
J^a\mathfrak{K}(M^\#)/J^b\mathfrak{K}(M^\#)\cong
(J^{-b+1}\KM/J^{-a+1}\KM)^\#.
\]
\end{lemma}

\section{Computing \texorpdfstring{$\mathcal{F}_i(M)$}{FiM} in rank two
via the generic kernel filtration}

The main theorem of this section shows that $\mathcal{F}_i(M)$ can be
computed on a subquotient of $M$ whose Loewy length is bounded in
terms of $i$. Again, we continue to require that the rank $r$ of $E$
is equal to two so that $E\cong\mathbb{Z}/p\times\mathbb{Z}/p$.

The following result will allow us to employ Lemma \ref{kernelcapimagelemma}
as the key in proving our main theorem.

\begin{lemma}\label{jranklemma}
Let $1\leq j\leq p$ and $i\leq j+1$. If $M$ is a $kE$-module having both
constant rank and constant $i$-rank, then the submodule $J^{-j}\KM$ also
has constant $i$-rank.
\end{lemma}

\begin{proof}
The statement is clear if $i=0$, so assume that $i\geq 1$.
Let $\overline{\alpha}\in\mathbb{P}^1(k)$ and observe that
if $m\in\Ker(X_\alpha^i,M)$, then $X_\alpha^{i-1}m\in\Ker(X_\alpha,M)$.
One then has $X_\alpha^{i-1}m\in\KM$ by Lemma
\ref{generickernelproperties} (2) so that $m\in J^{-(i-1)}\KM$ by Lemma
\ref{inversesofgenerickernels}. Combining this with the fact that
$J^{-(i-1)}\KM\subseteq J^{-j}\KM$ shows that
\[
\Ker(X_\alpha^i,J^{-j}\KM)=\Ker(X_\alpha^i,M).
\]
Because $M$ has constant $i$-rank, the dimension of the right hand term
is independent of $\overline{\alpha}$, hence so is the dimension of the
left hand term.
\end{proof}

\begin{proposition}\label{Fionsubmodule}
If $M$ is a $kE$-module of constant Jordan type and $i\leq j$, then
\[
\mathcal{F}_i(M)\cong\mathcal{F}_i(J^{-j}\KM).
\]
\end{proposition}

\begin{proof}
Let $N=J^{-j}\KM$. The inclusion $N\subseteq M$ induces an
inclusion of coherent sheaves $\widetilde{N}\subseteq\widetilde{M}$,
which in turn yields a natural inclusion
\begin{equation}\label{kernelcapimageinclusion}
\Ker\theta_N\cap\Img\theta_N^{i-1}\subseteq
\Ker\theta_M\cap\Img\theta_M^{i-1}.
\end{equation}
By Proposition \ref{jranklemma}, since $M$ has constant
Jordan type, $N$ has constant $(i-1)$-rank and constant $i$-rank.
Using Lemma \ref{fibresofinclusion}, this implies that the fibre of the
inclusion (\ref{kernelcapimageinclusion}) at any point
$\overline{\alpha}\in\mathbb{P}^1(k)$ is just the inclusion of vector
spaces
\[
\Ker(X_\alpha,N)\cap\Img(X_\alpha^{i-1},N)\subseteq
\Ker(X_\alpha,M)\cap\Img(X_\alpha^{i-1},M).
\]
But the latter inclusion is an equality for all $\overline{\alpha}$ by
Lemma \ref{kernelcapimagelemma}, so we actually have
\[
\Ker\theta_N\cap\Img\theta_N^{i-1}=\Ker\theta_M\cap\Img\theta_M^{i-1}.
\]
An identical argument also shows that
$\Ker\theta_N\cap\Img\theta_N^i=\Ker\theta_M\cap\Img\theta_M^i$,
hence
\[
\mathcal{F}_i(M)=\frac{\Ker\theta_M\cap\Img\theta_M^{i-1}}
{\Ker\theta_M\cap\Img\theta_M^i}=\frac{\Ker\theta_N\cap\Img\theta_N^{i-1}}
{\Ker\theta_N\cap\Img\theta_N^i}=
\mathcal{F}_i(N).\qedhere
\]
\end{proof}

\begin{theorem}\label{genkernelandFi}
If $M$ is a $kE$-module of constant Jordan type and
$i\leq\min\{j,\ell-1\}$, then
\[
\mathcal{F}_i(M)\cong\mathcal{F}_i(J^{-j}\KM/J^\ell\KM).
\]
\end{theorem}

\begin{proof}
We have $\mathcal{F}_i(M)=\mathcal{F}_i(J^{-j}\KM)$ by Proposition
\ref{Fionsubmodule}. Lemma \ref{vectorbundlesofduals} then tells us
that
$\mathcal{F}_i(M^\#)\cong\mathcal{F}_i((J^{-j}\KM)^\#)$. Note by Lemma
\ref{generickernelduality} that we have
\[
(J^{-j}\KM)^\#\cong M^\#/J^{j+1}\mathfrak{K}(M^\#),
\]
where the former module has constant $(i-1)$-rank, $i$-rank and
$(i+1)$-rank, hence the latter does as well.
Proposition \ref{Fionsubmodule} then shows that
\[
\mathcal{F}_i(M^\#/J^{j+1}\mathfrak{K}(M^\#))=
\mathcal{F}_i(J^{-\ell}\mathfrak{K}(M^\#)/J^{j+1}\mathfrak{K}(M^\#)).
\]
Putting this together now yields
\[
\mathcal{F}_i(M^\#)\cong
\mathcal{F}_i(J^{-\ell}\mathfrak{K}(M^\#)/J^{j+1}\mathfrak{K}(M^\#)).
\]
Using Lemma \ref{generickernelduality} then gives us
\[
\mathcal{F}_i(M^\#)\cong
\mathcal{F}_i((J^{-j}\KM/J^{\ell+1}\KM)^\#),
\]
and another use of Lemma \ref{vectorbundlesofduals} yields
\[
\mathcal{F}_i(M)\cong
\mathcal{F}_i(J^{-j}\KM/J^{\ell+1}\KM)
\]
as desired.
\end{proof}

The strongest form of Theorem \ref{genkernelandFi} is the following.

\begin{corollary}\label{stronggenkernelandFi}
If $M$ is a $kE$-module of constant Jordan type, then
\[
\mathcal{F}_i(M)\cong\mathcal{F}_i(J^{-i}\KM/J^{i+1}\KM).
\]
\end{corollary}

\begin{example}
If $M$ is a $kE$-module having the equal images property,
then for any $i\geq 1$, $\Rad^{i-1}(M)=J^{i-1}M$ also
has the equal images property, hence $\mathfrak{K}(J^{i-1}(M))=J^{i-1}M$.
It follows from Theorem \ref{eipradicalvb} and Corollary
\ref{stronggenkernelandFi} that
\[
\mathcal{F}_i(M)\cong\mathcal{F}_1(J^{i-1}M)\cong
\mathcal{F}_1(J^{i-1}M/J^{i+1}M).
\]
Note that the subquotient $J^{i-1}M/J^{i+1}M$ again has the equal images
property, and what's more, it has Loewy length at most two. One may verify
that such modules are isomorphic to direct sums of $W$-modules of the
form $W_{n,2}$. The techniques in Section \ref{Wmodulesection} may
therefore be applied in computing $\mathcal{F}_i(M)$ for any
module having the equal images property.
\end{example}

\section{The \texorpdfstring{$n\text{th}$}{nth} power generic kernel
and higher ranks}

It turns out that, although the generic kernel filtration is suitable for
detecting how the functors $\mathcal{F}_i$ behave with respect to a
$kE$-module $M$, there are even smaller subquotients of $M$ that
do a better job. In the most general setting the theory even carries over
to higher ranks. We begin our exposition in this broader context.

\begin{definition}
Let $E$ be an elementary abelian $p$-group of arbitrary rank $r$,
let $M$ be a $kE$-module, and fix $n>0$. For any dense open subset
$U\subseteq\mathbb{P}^{r-1}(k)$, let
\[
\leftindexp{U}{n}{M}=\sum_{\overline{\alpha}\in U}\Ker(X_\alpha^n,M).
\]
The \emph{$n$th power generic kernel} of $M$ is defined to be the
submodule
\[
\mathfrak{K}^n(M)=\bigcap_{\text{$U\subseteq\mathbb{P}^{r-1}(k)$ dense open}}
\leftindexp{U}{n}{M}
\]
of $M$.
\end{definition}

\begin{remark}
Our definition of the $n$th power generic kernel is a trivial extension
of that given in \cite{Carlson/Friedlander/Suslin:2011a} for the case $r=2$.
As was the case for generic kernels in rank two, because
$M$ is finite dimensional, we know that there always exists a dense open
subset $U\subseteq\mathbb{P}^{r-1}(k)$ for which
$\mathfrak{K}^n(M)=\leftindexp{U}{n}{M}$. If $M$ has constant $n$-rank,
the next proposition shows that one may take $U$ to be all of
$\mathbb{P}^{r-1}(k)$.
\end{remark}

\begin{proposition}\label{nthpowergenker}
If $M$ is a $kE$-module of constant $n$-rank, then
\[
\mathfrak{K}^n(M)=\leftindexp{\mathbb{P}^{r-1}(k)}{n}{M}=
\sum_{\overline{\alpha}\in\mathbb{P}^{r-1}(k)}\Ker(X_\alpha^n,M).
\]
\end{proposition}

\begin{proof}
We borrow the technique used in Proposition 7.6 of
\cite{Carlson/Friedlander/Suslin:2011a}.

Write $\mathfrak{K}^n(M)=\leftindexp{U}{n}{M}$ for some dense open
$U\subseteq\mathbb{P}^{r-1}(k)$. By the proof of Lemma 1.2 of
\cite{Friedlander/Pevtsova/Suslin:2007a}, the points
$\overline{\alpha}\in\mathbb{P}^{r-1}(k)$ for which $X_\alpha^n$ has
maximal rank on $\mathfrak{K}^n(M)$ also form a dense open subset of
$\mathbb{P}^{r-1}(k)$. These open subsets intersect non-trivially, hence
there exists a point $\overline{\alpha}\in\mathbb{P}^{r-1}(k)$ such that
$\mathfrak{K}^n(M)$ contains $\Ker(X_\alpha^n,M)$ and the rank of
$X_\alpha^n$ on $\mathfrak{K}^n(M)$ is maximal. For any point
$\overline{\beta}\in\mathbb{P}^{r-1}(k)$ we then have
\[
\begin{aligned}
\dim_k\Ker(X_\alpha^n,\mathfrak{K}^n(M)) &\leq
\dim_k\Ker(X_\beta^n,\mathfrak{K}^n(M)) \\
&\leq \dim_k\Ker(X_\beta^n,M) \\
&= \dim_k\Ker(X_\alpha^n,M) \\
&= \dim_k\Ker(X_\alpha^n,\mathfrak{K}^n(M)).
\end{aligned}
\]
Here the first inequality holds since $X_\alpha^n$ has maximal rank on
$\mathfrak{K}^n(M)$, the second inequality follows from the fact that
$\mathfrak{K}^n(M)$ is a submodule of $M$, the first equality holds because
$M$ has constant $n$-rank, and the second equality follows from the fact
that $\Ker(X_\alpha^n,M)$ is con-
tained in $\mathfrak{K}^n(M)$. In particular,
this shows that
$\dim_k\Ker(X_\beta^n,\mathfrak{K}^n(M))=\dim_k\Ker(X_\beta^n,M)$,
forcing us to have $\Ker(X_\beta^n,M)\subseteq\mathfrak{K}^n(M)$.
\end{proof}

We observe that the $n$th power generic kernel has a `dual' construction.

\begin{definition}
Let $E$ be an elementary abelian $p$-group having arbitrary rank $r$ and
let $M$ be a $kE$-module. We define the \emph{$n$th power generic image} of
$M$ to be the submodule
\[
\mathfrak{I}^n(M)=\bigcap_{\overline{\alpha}\in\mathbb{P}^{r-1}(k)}
\Img(X_\alpha^n,M)
\]
of $M$.
\end{definition}

Using the very same proof as that given in Proposition 8.4 of
\cite{Carlson/Friedlander/Suslin:2011a}, one readily establishes the
following, which shows how the $n$th power generic kernel and $n$th power
generic image are related via duality.

\begin{proposition}\label{genericnkernelduality}
For any $kE$-module $M$ we have
$\mathfrak{K}^n(M^\#)\cong(\mathfrak{I}^n(M))^\perp$.
\end{proposition}

\section{The \texorpdfstring{$n\text{th}$}{nth} power generic kernel and
the equal \texorpdfstring{$n$}{n}-images property}

We take a brief side trip here to explain what the $n$th power generic
kernel is, or rather, what it is not. The nomenclature would seem to
suggest that $\mathfrak{K}^n(M)$ might be the fibre of the generic operator
$\theta_M^n$ at any point $\overline{\alpha}\in\mathbb{P}^{r-1}(k)$,
but we shall see that this is not the case unless $M$ has a
very strong property. We first generalise Definition \ref{eipdefinition}
in the obvious way.

\begin{definition}
  A $kE$-module $M$ has the \emph{equal $n$-images property} if the image
  of $X_\alpha^n$ acting on $M$ is independent of the choice of $\overline{\alpha} \in
  \mathbb{P}^{r-1} ( k )$. 
\end{definition}

We now give a general lemma.

\begin{lemma}\label{incllemma}
Let $X$ be a variety and $\mathcal{E}$ a vector bundle on $X$.
Let $\mathcal{E}_1$ and $\mathcal{E}_2$ be subbundles of $\mathcal{E}$
such that $\mathcal{E}/\mathcal{E}_{1}$ and $\mathcal{E}/\mathcal{E}_{2}$
are both locally free. Then $\mathcal{E}_{1}$ is a subbundle of
$\mathcal{E}_{2}$ inside $\mathcal{E}$ if and only if
$\mathcal{E}_{1}\otimes k( x ) \subseteq \mathcal{E}_{2} \otimes k ( x )$
for each closed point $x \in X$.
\end{lemma}

\begin{proof}
  Suppose first that $\mathcal{E}_{1} \subseteq \mathcal{E}_{2}$. By the universal property of cokernels,
  there is a unique morphism $\mathcal{E}/\mathcal{E}_{1} \rightarrow \mathcal{E}/\mathcal{E}_{2}$ making the
  following diagram commute.
  \[
    \xymatrix{
    0 \ar@{>}[r] & \mathcal{E}_1 \ar@{>}[r] \ar@{>}[d] & \mathcal{E} \ar@{>}[r]\ar@{=}[d] & \mathcal{E}/\mathcal{E}_1 \ar@{>}[r]\ar@{>}[d]& 0 \\
    0 \ar@{>}[r] & \mathcal{E}_2 \ar@{>}[r]            & \mathcal{E} \ar@{>}[r]           & \mathcal{E}/\mathcal{E}_2 \ar@{>}[r] & 0 
  }
  \]
  Since all sheaves are locally free, both rows remain exact after
  tensoring with $k ( x )$ for any $x\in X$. The fibre of the left vertical arrow is therefore
  an inclusion $\mathcal{E}_{1} \otimes k ( x ) \subseteq \mathcal{E}_{2} \otimes k ( x )$.
  Conversely, suppose that $\mathcal{E}_{1} \otimes k ( x )
  \subseteq \mathcal{E}_{2} \otimes k ( x )$ for all $x \in X$ and let $\pi$
  be the composition
  $\mathcal{E}_{1} \hookrightarrow \mathcal{E} \rightarrow \mathcal{E}/\mathcal{E}_{2}$. 
  After base changing to $k ( x )$, we have an exact sequence
  \[
    \xymatrix{
    \Ker\pi \otimes k ( x )  \ar@{>}[r] & 
    \mathcal{E}_{1} \otimes k ( x ) \ar@{>}[r] & \Img\pi \otimes k ( x )
    \ar@{>}[r] & 0  }.
  \]
  The map $\pi \otimes k ( x )$ is the composition
  \[
    \xymatrix{
    \mathcal{E}_{1} \otimes k ( x ) \ar@{>}[r] & 
    \Img\pi \otimes k (x )  \ar@{^{(}->}[r] & ( \mathcal{E}/\mathcal{E}_{2} ) \otimes k ( x ) \cong \dfrac{\mathcal{E} \otimes
   k ( x )}{\mathcal{E}_{2} \otimes k ( x )}   }.
  \]
  By our assumption, the first map is zero for all $x \in X$, hence
  $\Img\pi \otimes k ( x ) =0$ for all closed points $x \in
  X$. Since $\Img\pi$ is coherent, we have $\Img\pi =0$ so that $\mathcal{E}_{1} \subseteq \mathcal{E}_{2}$. 
\end{proof}

\begin{proposition}
  If $M$ is a $kE$-module of constant Jordan type, then
  $\Ker\theta^{n}_{M} \subseteq\widetilde{\mathfrak{K}^n(M)}$.
  Furthermore, equality holds if and only if $M$ has the equal $n$-images
  property. 
\end{proposition}

\begin{proof}
  First note that if $M$ has constant Jordan type, then $M$ has constant
  $n$-rank for all $n>0$. It follows from Proposition \ref{nthpowergenker}
  that $\mathfrak{K}^n ( M ) = \sum_{\overline{\alpha}\in\mathbb{P}^{r-1}(k)}\Ker(X^n_{\alpha},M)$.
  
It is obvious that $\Ker\theta_M^n$ is a subsheaf of $\widetilde{M}$.
  The inclusion $\mathfrak{K}^{n} ( M ) \subseteq M$ also
  identifies $\widetilde{\mathfrak{K}^{n} ( M )}$
  as a subsheaf of $\widetilde{M}$. Since $M$ has constant Jordan type,
  $\Img\theta_M^n$ is locally free. The sheaf
  $\widetilde{M/\mathfrak{K}^n(M)}$ is also locally free, so
  $\Ker\theta_M^n$ and $\widetilde{\mathfrak{K}^n(M)}$ both satisfy the
  initial hypothe-
  ses of Lemma \ref{incllemma}.
  
To show that
  $\Ker\theta^{n}_{M} \subseteq \widetilde{\mathfrak{K}^n(M)}$
  it suffices, using Lemma \ref{incllemma}, to prove the inclusion on fibres.
  But this is clear since
\[
\Ker\theta^{n}_{M} \otimes k ( \overline{\alpha} ) = \Ker( X^{n}_{\alpha}, M)
\subseteq\sum_{\overline{\beta}\in\mathbb{P}^{r-1}(k)}\Ker(X_\beta^n,M)=
  \mathfrak{K}^{n} ( M ).
\]
This establishes the first statement.
  
Now suppose that $\Ker\theta_{M}^{n}  =  \widetilde{\mathfrak{K}^{n} ( M )}$.
We have
\[
\Ker(X_\alpha^n,M)=\Ker\theta_{M}^{n} \otimes  k( \overline{\alpha} ) =
\mathfrak{K}^{n} ( M ) =\sum_{\overline{\beta}\in\mathbb{P}^{r-1}(k)}\Ker(X_\beta^n,M)
\]
  for all $\overline{\alpha}   \in  
  \mathbb{P}^{r-1}(k)$, which shows that $M$ has the equal $n$-images
  property.
  
Conversely, if $M$ has the equal $n$-images property, then
  $\Ker\theta_{M}^{n} \otimes 
  k( \overline{\alpha} ) = \mathfrak{K}^{n} ( M )$ for all
  $\overline{\alpha}   \in 
  \mathbb{P}^{r-1}(k)$. The reverse implication in Lemma \ref{incllemma}
  then shows that $\Ker\theta_{M}^{n} \subseteq\widetilde{\mathfrak{K}^n(M)}$.
  Since these are vector bundles of the
  same rank with equal fibres, we must have $\Ker\theta_{M}^{n} =
  \widetilde{\mathfrak{K}^{n} ( M )}$.
\end{proof}

\section{Computing \texorpdfstring{$\mathcal{F}_i(M)$}{FiM} using
\texorpdfstring{$n\text{th}$}{nth} power generic kernels}

In this section we show that the $n$th power generic kernels and
$n$th power generic images of a $kE$-module $M$ can be used to
compute the vector bundles $\mathcal{F}_i(M)$ for a $kE$-module
$M$ of constant Jordan type. Again, our discussion applies to
the case where $E$ has arbitrary rank $r$.

\begin{lemma}\label{jrankskernellemma}
Let $M$ be a $kE$-module and $n\geq 0$. Then for all
$\overline{\alpha}\in\mathbb{P}^{r-1}(k)$ and all $j\leq n$,
the $j$-rank of $X_\alpha$ on $\mathfrak{K}^n(M)$ is equal to
the $j$-rank of $X_\alpha$ on $M$.
\end{lemma}

\begin{proof}
Since $j\leq n$, we have
$\Ker(X_\alpha^j,M)\subseteq\Ker(X_\alpha^n,M)\subseteq\mathfrak{K}^n(M)$.
It immediately follows that
$\Ker(X_\alpha^j,\mathfrak{K}^n(M))=\Ker(X_\alpha^j,M)$,
thus $\rank(X_\alpha^j,\mathfrak{K}^n(M))=\rank(X_\alpha^j,M)$
by the rank-nullity theorem.
\end{proof}

\begin{proposition}\label{Knhasconstantranks}
If $j\leq n$, then $\mathfrak{K}^n(M)$ has constant $j$-rank if and
only if $M$ has constant $j$-rank.
\end{proposition}

\begin{proof}
Immediate from the lemma.
\end{proof}

The following should be compared with Lemma \ref{kernelcapimagelemma}.

\begin{lemma}\label{Knkernelcapimagethesame}
Let $M$ be a $kE$-module having constant $n$-rank for some $n\geq 0$. Then
for all $i\leq n-1$ and all $\,\overline{\alpha}\in\mathbb{P}^{r-1}(k)$
we have
\[
\Ker(X_\alpha,\mathfrak{K}^n(M))\cap\Img(X_\alpha^i,\mathfrak{K}^n(M))=
\Ker(X_\alpha,M)\cap\Img(X_\alpha^i,M).
\]
\end{lemma}

\begin{proof}
The rightwards containment is obvious since $\mathfrak{K}^n(M)$ is a
submodule of $M$.
For the reverse containment, note that if $m\in M$ satisfies
$X_\alpha m=0$ and there exists $m'\in M$ such that $X_\alpha^i m'=m$,
then $m'\in\Ker(X_\alpha^{i+1},M)$. Using the characterisation
\[
\mathfrak{K}^n(M)=
\sum_{\overline{\beta}\in\mathbb{P}^{r-1}(k)}\Ker(X_\beta^n,M),
\]
this shows that $m'\in\mathfrak{K}^n(M)$ so that
$m\in\Img(X_\alpha^i,\mathfrak{K}^n(M))$.
\end{proof}

\begin{proposition}\label{Fiongenericnkernel}
If $M$ is a $kE$-module having constant $j$-rank for all $j\leq n$, then
for all $i\leq n-1$ we have
$\mathcal{F}_i(M)=\mathcal{F}_i(\mathfrak{K}^n(M))$.
\end{proposition}

\begin{proof}
By Proposition \ref{Knhasconstantranks}, $\mathfrak{K}^n(M)$ has
constant $(i-1)$-rank, $i$-rank and $(i+1)$-rank. For $j=i-1$ and $i$,
it follows by Proposition \ref{fibresofinclusion} that the fibres of the
inclusions
\[
\Ker\theta_{\mathfrak{K}^n(M)}\cap\Img\theta_{\mathfrak{K}^n(M)}^j
\subseteq\Ker\theta_M\cap\Img\theta_M^j
\]
at any point $\overline{\alpha}\in\mathbb{P}^{r-1}(k)$ are the inclusions
of vector spaces
\[
\Ker(X_\alpha,\mathfrak{K}^n(M))\cap\Img(X_\alpha^j,\mathfrak{K}^n(M))
\subseteq\Ker(X_\alpha,M)\cap\Img(X_\alpha^j,\theta_M).
\]
By Lemma \ref{Knkernelcapimagethesame}, these inclusions are equalities,
and because all of the above sheaves are locally free
(see Lemma \ref{kernelcapimage}), we conclude that
\[
\Ker\theta_{\mathfrak{K}^n(M)}\cap\Img\theta_{\mathfrak{K}^n(M)}^j
=\Ker\theta_M\cap\Img\theta_M^j.
\]
We therefore have
\[
\mathcal{F}_i(M)=
\frac{
\Ker\theta_M\cap\Img\theta_M^{i-1}
}
{
\Ker\theta_M\cap\Img\theta_M^i
}=
\frac{
\Ker\theta_{\mathfrak{K}^n(M)}\cap\Img\theta_{\mathfrak{K}^n(M)}^{i-1}
}
{
\Ker\theta_{\mathfrak{K}^n(M)}\cap\Img\theta_{\mathfrak{K}^n(M)}^i
}=\mathcal{F}_i(\mathfrak{K}^n(M)).
\qedhere
\]
\end{proof}

In light of the duality that exists between
$\mathfrak{K}^n$ and $\mathfrak{I}^n$, we now show that quotienting out
by $\mathfrak{I}^n(M)$ has the same effect as taking the
submodule $\mathfrak{K}^n(M)$.

\begin{corollary}\label{FionMmodgenericnimage}
If $M$ is a $kE$-module that has constant $j$-rank for all $j\leq n$,
then for all $i\leq n-1$ we have
$\mathcal{F}_i(M)=\mathcal{F}_i(M/\mathfrak{I}^n(M))$.
\end{corollary}

\begin{proof}
By Lemma \ref{vectorbundlesofduals}, Theorem \ref{Fiongenericnkernel}
and Proposition \ref{genericnkernelduality}, respectively, we compute
\[
\begin{aligned}
\mathcal{F}_i(M) &\cong \mathcal{F}_i(M^\#)^\vee(-i+1)=
\mathcal{F}_i(\mathfrak{K}^n(M^\#))^\vee(-i+1)\cong
\mathcal{F}_i(\mathfrak{I}^n(M)^\perp)^\vee(-i+1) \\
&\cong \mathcal{F}_i((M/\mathfrak{I}^n(M))^\#)^\vee(-i+1)\cong
\mathcal{F}_i(M/\mathfrak{I}^n(M)).
\end{aligned}\qedhere
\]
\end{proof}

Before presenting the main theorem of the section, we give a somewhat
obvious lemma regarding the relationship between generic $n$-kernels
and generic $n$-images. 

\begin{lemma}\label{jranksimagelemma}
Let $M$ be a $kE$-module and $n\geq 0$. Then for all
$\overline{\alpha}\in\mathbb{P}^{r-1}(k)$ and all $j\leq n$,
the $j$-rank of $X_\alpha$ on $M/\mathfrak{I}^n(M)$ is equal to
the $j$-rank of $X_\alpha$ on $M$.
\end{lemma}

\begin{proof}
This follows from Lemma \ref{jrankskernellemma} and the duality formula
\ref{genericnkernelduality}.
\end{proof}

We are now ready to prove our main result.

\begin{theorem}
If $M$ has constant Jordan type and $i\leq\min\{n-1,m-1\}$, then
\[
\mathcal{F}_i(M)\cong
\mathcal{F}_i(\mathfrak{K}^n(M)/\mathfrak{I}^m\mathfrak{K}^n(M))
\qquad\text{and}\qquad
\mathcal{F}_i(M)\cong\mathcal{F}_i(\mathfrak{K}^n(M/\mathfrak{I}^m(M))).
\]
\end{theorem}

\begin{proof}
By Lemmas \ref{jrankskernellemma} and \ref{jranksimagelemma}, every
submodule and every quotient in the above formul{\ae} has constant
$j$-rank for all $j\leq\min\{n,m\}$. The proof therefore follows by
the appropriate use of Proposition \ref{Fiongenericnkernel} and
Corollary \ref{FionMmodgenericnimage}.
\end{proof}

Again, we specify the above theorem in its strongest form.

\begin{corollary}
If $M$ has constant Jordan type, then
\[
\mathcal{F}_i(M)\cong
\mathcal{F}_i(\mathfrak{K}^{i+1}(M)/\mathfrak{I}^{i+1}\mathfrak{K}^{i+1}(M))
\qquad\text{and}\qquad
\mathcal{F}_i(M)\cong
\mathcal{F}_i(\mathfrak{K}^{i+1}(M/\mathfrak{I}^{i+1}(M))).
\]
\end{corollary}

We show in the following section how these results may be applied to
the rank two case in better understanding how to compute the bundles
$\mathcal{F}_i(M)$ by restricting to subquotients of $M$ having smaller
Loewy length.

\section{Some examples and applications}

It is now time to apply the somewhat technical results of the previous
sections to some elementary, concrete examples. We shall focus our
attention on the case $r=2$, which in some sense is easier to work with.

It was shown in \cite{Carlson/Friedlander/Suslin:2011a} that if
$M$ has constant rank, then $\mathfrak{K}^n(M)$ is contained in 
$J^{-n+1}\KM$. This means that we have a series of inclusions
\[
\begin{array}{ccccccccc}
\KM & \subseteq & J^{-1}\KM & \subseteq & J^{-2}\KM & \subseteq & \cdots &
\subseteq & J^{-p+1}\KM \\
\rotatebox{90}{$=\,$} & & \rotatebox{90}{$\subseteq\,$} & &
\rotatebox{90}{$\subseteq\,$} & & & & \rotatebox{90}{$=\,$} \\
\mathfrak{K}^1(M) & \subseteq & \mathfrak{K}^2(M) & \subseteq &
\mathfrak{K}^3(M) & \subseteq & \cdots &
\subseteq & \mathfrak{K}^p(M)=M
\end{array}
\]
and dually
\[
\begin{array}{ccccccccc}
J\KM & \supseteq & J^2\KM & \supseteq & J^3\KM & \supseteq & \cdots &
\supseteq & J^p\KM \vspace{-0.1in}\\
\rotatebox{-90}{$=\ $} & & \rotatebox{270}{$\subseteq$} & &
\rotatebox{270}{$\subseteq$} & & & & \rotatebox{270}{$=$} \\
\mathfrak{I}(M) & \supseteq & \mathfrak{I}^2(M) & \supseteq &
\mathfrak{I}^3(M) & \supseteq & \cdots &
\supseteq & \mathfrak{I}^p(M)=0
\end{array}
\]
In general, the vertical inclusions can be strict, i.e.,
$\mathfrak{K}^n(M)/\mathfrak{I}^m\mathfrak{K}^n(M)$ and
$\mathfrak{K}^n(M/\mathfrak{I}^m(M))$
tend to be strictly smaller subquotients of $J^{-n+1}\KM/J^m\KM$.

\begin{example}
Let $p=3$ and $M$ the $kE$-module given by the following diagram.
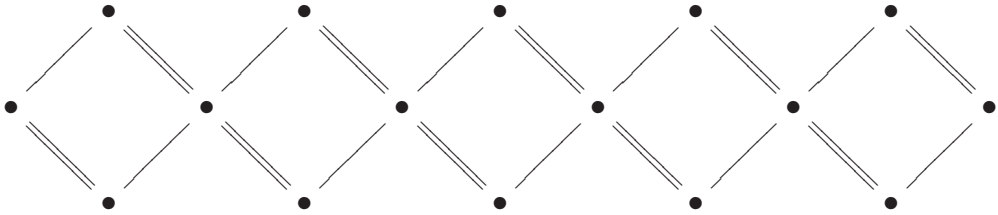
\[
\xymatrix{
& \bullet \ar@{-}[dl] \ar@{=}[dr] & & \bullet \ar@{-}[dl] \ar@{=}[dr] & &
\bullet \ar@{-}[dl] \ar@{=}[dr] & & \bullet \ar@{-}[dl] \ar@{=}[dr] & &
\bullet \ar@{-}[dl] \ar@{=}[dr] \\
\bullet \ar@{=}[dr] & & \bullet \ar@{-}[dl] \ar@{=}[dr] & &
\bullet \ar@{-}[dl] \ar@{=}[dr] & & \bullet \ar@{-}[dl] \ar@{=}[dr] & &
\bullet \ar@{-}[dl] \ar@{=}[dr] & & \bullet \ar@{-}[dl] \\
& \bullet & & \bullet & & \bullet & & \bullet & & \bullet
}
\]
One may calculate that $M$ is a module of constant Jordan type
$[3]^4[2]^2$. Moreover,
\[
M=J^{-1}\KM/J^2\KM.
\]
On the other hand, $\mathfrak{K}^2(M)$ is given by the diagram
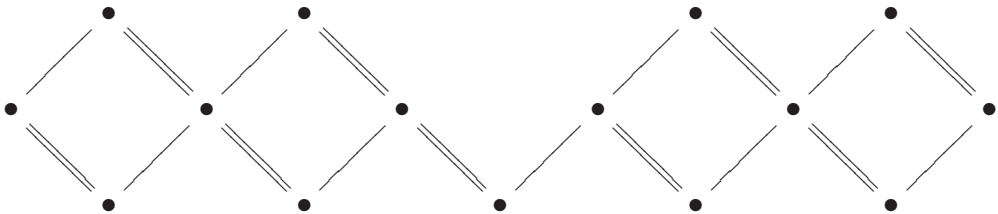
\[
\xymatrix{
& \bullet \ar@{-}[dl] \ar@{=}[dr] & & \bullet \ar@{-}[dl] \ar@{=}[dr] & &
& & \bullet \ar@{-}[dl] \ar@{=}[dr] & &
\bullet \ar@{-}[dl] \ar@{=}[dr] \\
\bullet \ar@{=}[dr] & & \bullet \ar@{-}[dl] \ar@{=}[dr] & &
\bullet \ar@{-}[dl] \ar@{=}[dr] & & \bullet \ar@{-}[dl] \ar@{=}[dr] & &
\bullet \ar@{-}[dl] \ar@{=}[dr] & & \bullet \ar@{-}[dl] \\
& \bullet & & \bullet & & \bullet & & \bullet & & \bullet
}
\]
and $\mathfrak{K}^2(M)/\mathfrak{I}^2\mathfrak{K}^2(M)$ has diagram
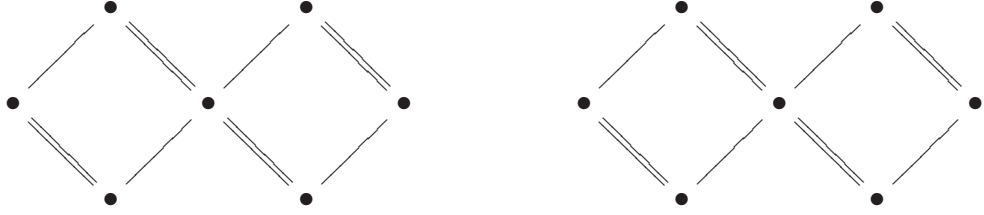
\[
\xymatrix{
& \bullet \ar@{-}[dl] \ar@{=}[dr] & & \bullet \ar@{-}[dl] \ar@{=}[dr] & &
& & \bullet \ar@{-}[dl] \ar@{=}[dr] & &
\bullet \ar@{-}[dl] \ar@{=}[dr] \\
\bullet \ar@{=}[dr] & & \bullet \ar@{-}[dl] \ar@{=}[dr] & &
\bullet \ar@{-}[dl] & & \bullet \ar@{=}[dr] & &
\bullet \ar@{-}[dl] \ar@{=}[dr] & & \bullet \ar@{-}[dl] \\
& \bullet & & \bullet & & & & \bullet & & \bullet
}
\]
One may check that the latter diagram is also that for
$\mathfrak{K}^2(M/\mathfrak{I}^2(M))$.
\end{example}

It is easy to check that both constructions
$\mathfrak{K}^n(-)/\mathfrak{I}^m\mathfrak{K}^n(-)$ and
$\mathfrak{K}^n(-/\mathfrak{I}^m(-))$ are functorial. 
We do not know of an example in which they are not naturally isomorphic.

\begin{question}
If $M$ is any $kE$-module, is it always the case that
\[
\mathfrak{K}^n(M)/\mathfrak{I}^m\mathfrak{K}^n(M)\cong
\mathfrak{K}^n(M/\mathfrak{I}^m(M))?
\]
If not, then is the statement true whenever $M$ has constant Jordan type,
for example?
\end{question}

We now show how the results of the paper may help us compute
$\mathcal{F}_1(M)$ in an extremely efficient way.

\begin{example}\label{mainexample}
Consider the seven dimensional $kE$-module $M$ given by the diagram
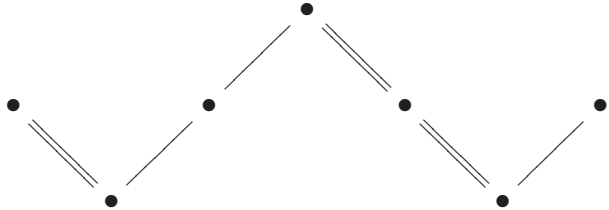
\[
\xymatrix{
& & & \bullet \ar@{-}[dl] \ar@{=}[dr] & & & \\
\bullet \ar@{=}[dr] & & \bullet \ar@{-}[dl] & & \bullet \ar@{=}[dr] & &
\bullet \ar@{-}[dl]\\
& \bullet & & & & \bullet &
}
\]
This first appeared in \cite{Carlson/Friedlander/Suslin:2011a}. To
show how the results of this paper make the computation of $\mathcal{F}_1(M)$
almost trivial, note that $\mathfrak{K}^2(M)$ is the submodule given by
the diagram
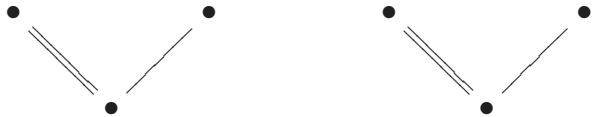
\[
\xymatrix{
\bullet \ar@{=}[dr] & & \bullet \ar@{-}[dl] & & \bullet \ar@{=}[dr] & &
\bullet \ar@{-}[dl]\\
& \bullet & & & & \bullet &
}
\]
By the discussion in Section \ref{Wmodulesection}, one sees that this
is just two copies of the $W$-module $W_{2,2}$.
It now follows from Proposition \ref{Fiongenericnkernel} and
Theorem \ref{wmodulevectorbundles} that
\[
\mathcal{F}_1(M)=\mathcal{F}_1(\mathfrak{K}^2(M))\cong
\mathcal{F}_1(W_{2,2})\oplus\mathcal{F}_1(W_{2,2})\cong
\mathcal{O}_{\mathbb{P}^1}(-1)\oplus\mathcal{O}_{\mathbb{P}^1}(-1).
\]

Dually, consider the $k$-linear dual $M^\#$, whose structure
is given by the dual diagram
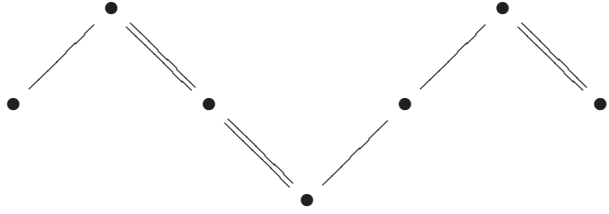
\[
\xymatrix{
& \bullet \ar@{-}[dl] \ar@{=}[dr] & & & & \bullet \ar@{-}[dl] \ar@{=}[dr]
& \\
\bullet & & \bullet \ar@{=}[dr] & & \bullet \ar@{-}[dl] & & \bullet \\
& & & \bullet & & &
}
\]
In this case, $\mathfrak{K}^2(M/\mathfrak{I}^2(M))\cong M/\mathfrak{I}^2(M)$,
which is given by the diagram
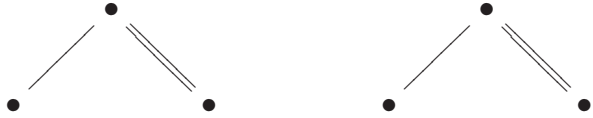
\[
\xymatrix{
& \bullet \ar@{-}[dl] \ar@{=}[dr] & & & & \bullet \ar@{-}[dl] \ar@{=}[dr]
& \\
\bullet & & \bullet & & \bullet & & \bullet \\
}
\]
This is now two copies of the dual $W$-module $W_{n,2}^\#$. From this we
see that
\[
\mathcal{F}_1(M^\#)=\mathcal{F}_1(M/\mathfrak{I}^2(M))\cong
\mathcal{F}_1(W_{2,2}^\#)\oplus\mathcal{F}_1(W_{2,2}^\#)\cong
\mathcal{O}_{\mathbb{P}^1}(1)\oplus\mathcal{O}_{\mathbb{P}^1}(1).
\]
\end{example}

\section{A discussion about vector bundles in rank two}

Let $E=\mathbb{Z}/p\times\mathbb{Z}/p$ and recall that for $p>2$, $kE$
has wild representation type. What's more, it was shown by Benson
that the subcategory $\cJt(kE)$ of modules of constant Jordan type
is also wild. On the other hand, we know that the functor
$\mathcal{F}_1\colon\cJt(kE)\rightarrow\vect(\mathbb{P}^1(k))$ is
essentially surjective, and the right hand side is certainly tame by
Grothendieck's theorem. One of the original goals of this project
was to find a structural invariant $G(M)$ of $M$ landing in
some tame subcategory of $\fgmod(kE)$ such that
$\mathcal{F}_1(M)\cong\mathcal{F}_1(G(M))$. The functorial relationship
between modules of constant Jordan type and vector bundles would then
reduce to one between two tame categories.

It was shown in Corollary \ref{stronggenkernelandFi} that
$\mathcal{F}_1(M)\cong\mathcal{F}_1(J^{-1}\KM/J^2\KM)$, so one might
hope that the modules $M$ for which $J^{-1}\KM=M$ and $J^2\KM=0$
form a tame subcategory, or at least those $M$ of the form
$J^{-1}\mathfrak{K}(N)/J^2\mathfrak{K}(N)$ for some module $N$ of
constant Jordan type. This
might also appear somewhat plausible due to the fact that such modules have
Loewy length only three.

Unfortunately, Benson has also shown that the category of modules $M$
having constant Jordan type satisfying $J^{-1}\KM=M$ and $J^2\KM=0$
is in fact wild. (See Section 5 of \cite{Benson:book} for details.)
Our investigation has led us to consider
the functors $\mathfrak{K}^2(-)/\mathfrak{I}^2\mathfrak{K}^2(-)$
and $\mathfrak{K}^2(-/\mathfrak{I}^2(-))$ as more likely candidates.
As can be seen from Example \ref{mainexample}, what appears to happen
is that both $\mathfrak{K}^2(M)/\mathfrak{I}^2\mathfrak{K}^2(M)$ and
$\mathfrak{K}^2(M/\mathfrak{I}^2(M))$ break up into direct sums of
$W$-modules (or duals of $W$-modules) having Loewy length at most two, the
decomposition of which allows one to instantly calculate $\mathcal{F}_1(M)$
using Theorem \ref{wmodulevectorbundles}. Although these sub-
quotients might still contain direct summands of Loewy length three,
we conjecture that such summands contribute nothing to $\mathcal{F}_1(M)$.
We make this more precise.

\begin{conjecture}
If $M$ is a $k(\mathbb{Z}/p)^2$-module of constant Jordan type and $N$ is
any inde-
composable direct summand of
$\mathfrak{K}^2(M)/\mathfrak{I}^2\mathfrak{K}^2(M)$
or $\mathfrak{K}^2(M/\mathfrak{I}^2(M))$ that has Loewy length three, then
$\mathcal{F}_1(N)$ is the zero sheaf on $\mathbb{P}^1(k)$.
\end{conjecture}

\bibliographystyle{amsplain}
\bibliography{repcoh.bib}

\end{document}